\documentclass[12pt]{amsart}

\newif\ifPDF
\ifx\pdfoutput\undefined\PDFfalse
\else \ifnum \pdfoutput > 0 \PDFtrue
        \else \PDFfalse
        \fi
\fi

\usepackage[centertags]{amsmath}
\usepackage{amsfonts}
\usepackage{mathrsfs}
\usepackage{textcomp}
\usepackage{amssymb}
\usepackage{amsthm}
\usepackage{newlfont}
\usepackage[all]{xy}

\ifPDF	
  \usepackage[pdftex]{color, graphicx}
  \usepackage[pdftex, bookmarks, colorlinks]{hyperref}
  \hypersetup{colorlinks=false}

\else
  \usepackage{color}
  \usepackage[dvips]{graphicx}
  \usepackage[dvips]{hyperref}
\fi

\usepackage[margin=1in]{geometry}

\usepackage[pagewise, mathlines, displaymath]{lineno}

\newtheorem{thm}{Theorem}[section]
\newtheorem{cor}[thm]{Corollary}
\newtheorem{lem}[thm]{Lemma}
\newtheorem{prop}[thm]{Proposition}
\theoremstyle{definition}
\newtheorem{defn}[thm]{Definition}
\theoremstyle{remark}
\newtheorem{rem}[thm]{Remark}

\numberwithin{equation}{section}

\newcommand{\norm}[1]{\left\Vert#1\right\Vert}
\newcommand{\abs}[1]{\left\vert#1\right\vert}

\newcommand{\Real}{\mathbb R}
\newcommand{\Int}{\mathbb Z}
\newcommand{\Comp}{\mathbb C}

\newcommand{\eps}{\varepsilon}

\newcommand{\Kzero}{\mathrm{K}_0}
\newcommand{\Kone}{\mathrm{K}_1}

\begin{document}


\title{$\mathcal Z$-stability of $\mathrm{C}(X)\rtimes\Gamma$}

\author{Zhuang Niu}
\address{Department of Mathematics, University of Wyoming, Laramie, WY, 82071, USA}
\email{zniu@uwyo.edu}

\thanks{The research is supported by an NSF grant (DMS-1800882).}

\date{\today}


\begin{abstract}
Let $(X, \Gamma)$ be a free and minimal topological dynamical system, where $X$ is a separable compact Hausdorff space and $\Gamma$ is a countable infinite discrete amenable group. It is shown that if $(X, \Gamma)$ has the Uniform Rokhlin Property and Cuntz comparison of open sets, then $\mathrm{mdim}(X, \Gamma)=0$ implies that $(\mathrm{C}(X) \rtimes\Gamma)\otimes\mathcal Z \cong \mathrm{C}(X) \rtimes\Gamma$, where $\mathrm{mdim}$ is the mean dimension and $\mathcal Z$ is the Jiang-Su algebra. In particular, in this case, $\mathrm{mdim}(X, \Gamma)=0$ implies that the C*-algebra $\mathrm{C}(X) \rtimes\Gamma$ is classified by the Elliott invariant. 
\end{abstract}

\maketitle

\section{Introduction}

Let $\Gamma$ be a discrete amenable group, and let $(\Omega, \mu)$ be a $\sigma$-finite standard measure space. Let $(\Omega, \mu) \curvearrowleft \Gamma$ be a free and ergodic action with absolutely continuous finite invariant measure. By the classification of injective von Neumann algebras, it is well known that the von Neumann II$_1$-factor $\mathrm{L}^\infty(\Omega, \mu)\rtimes\Gamma$ is isomorphic to the unique hyperfinite II$_1$-factor $R$. Thus, all such crossed products $\mathrm{L}^\infty(\Omega, \mu)\rtimes\Gamma$ are isomorphic.

In the topological setting, consider a compact separable Hausdorff space $X$, and consider a minimal and free action $X\curvearrowleft \Gamma$. Then the crossed product C*-algebra $\mathrm{C}(X)\rtimes \Gamma$ is simple separable unital  nuclear and satisfies the UCT. Thus it is a very natural object for the Elliott's classification program of nuclear C*-algebras.

Many efforts have been devoted to the classifiability of  $\mathrm{C}(X)\rtimes\Gamma$ (in term of the K-theoretical Elliott invariant); see, for instance, \cite{Put-PJM}, \cite{QL-Ph-min-diff} \cite{LP-Dym}, \cite{TW-Dym-1}, \cite{Szabo-Z}, \cite{Karen-sphere}, \cite{Winter-TA}, etc. However, as shown by Giol and Kerr in \cite{GK-Dyn}, there exist minimal and free actions $X\curvearrowleft \mathbb Z$ such that the C*-algebras $A=\mathrm{C}(X) \rtimes \Int$ cannot be classified by the Elliott invariant, and these C*-algebras do not absorb the Jiang-Su algebra $\mathcal Z$ tensorially (i.e., $A \otimes\mathcal Z \ncong A$).

The dynamical systems constructed in \cite{GK-Dyn} have non-zero mean (topological) dimension; and in \cite{EN-MD0}, it is shown that if a minimal and free $\mathbb Z$-action has zero mean dimension (this particularly includes all strictly ergodic systems and all minimal dynamical systems with finite topological entropy, see \cite{Lind-MD}), then the C*-algebra $\mathrm{C}(X) \rtimes \Int$ must be $\mathcal Z$-absorbing and is classifiable (see \cite{ENST-ASH} and \cite{EGLN-ASH}).

In this note, one considers an arbitrary discrete amenable group $\Gamma$, and studies the $\mathcal Z$-stability of $\mathrm{C}(X) \rtimes\Gamma$. Under the assumption that $(X, \Gamma)$ has the Uniform Rokhlin Property (URP) and Cuntz comparison of open sets (COS), which are introduced in \cite{Niu-MD-Z}, one has that $\mathrm{mdim}(X, \Gamma) = 0$ implies that $(\mathrm{C}(X) \rtimes\Gamma)\otimes\mathcal Z \cong \mathrm{C}(X) \rtimes\Gamma$, where $\mathrm{mdim}$ is the mean dimension. 
In particular, this implies that $\mathrm{C}(X) \rtimes\Gamma$ is classified by its Elliott invariant.

Recall
\begin{defn}[Definition 3.1 and Definition 4.1 of \cite{Niu-MD-Z}]
A topological dynamical system $(X, \Gamma)$, where $\Gamma$ is a discrete amenable group, is said to have Uniform Rokhlin Property (URP) if for any $\eps>0$ and any finite set $K\subseteq \Gamma$, there exist closed sets $B_1, B_2, ..., B_S \subseteq X$ and $(K, \eps)$-invariant sets $\Gamma_1, \Gamma_2, ..., \Gamma_S \subseteq \Gamma$ such that
$$B_s\gamma,\quad \gamma\in \Gamma_s,\ s=1, ..., S, $$
are mutually disjoint and
$$\mathrm{ocap}(X\setminus\bigsqcup_{s=1}^S\bigsqcup_{\gamma\in \Gamma_s}B_s\gamma) < \eps,$$
where $\mathrm{ocap}$ denote the orbit capacity (see, for instance, Definition 5.1 of \cite{Lindenstrauss-Weiss-MD}).

The dynamical system $(X, \Gamma)$ is said to have $(\lambda, m)$-Cuntz-comparison of open sets, where $\lambda\in (0, 1]$ and $m\in \mathbb N$, if for any open sets $E, F\subseteq X$ with $$ \mu(E) < \lambda \mu(F),\quad \mu \in\mathcal M_1(X, \Gamma),$$ 
where $\mathcal{M}_1(X, \Gamma)$ is the simplex of all invariant probability measures on $X$, then
$$\varphi_E \precsim \underbrace{\varphi_F\oplus\cdots \oplus \varphi_F}_m\quad\mathrm{in}\ \mathrm{C}(X)\rtimes\Gamma,$$ 
where $\varphi_E$ and $\varphi_F$ are continuous functions supporting on $E$ and $F$ respectively.

The dynamical system $(X, \Gamma)$ is said to have Cuntz comparison of open sets (COS) if it has $(\lambda, m)$-Cuntz-comparison on open sets for some $\lambda$ and $m$.
\end{defn}

The properties of (URP) and (COS) seem to hold for all free and minimal actions by a finitely generated discrete amenable group: any free minimal $\mathbb Z^d$-action has the (URP) and has $(\frac{1}{4}, (2\lfloor\sqrt{d}\rfloor + 1)^d +1)$-Cuntz-comparison of open sets (\cite{Niu-MD-Zd}); any free and minimal $\Gamma$-action has the (URP) and has $(\frac{1}{4}, 1)$-Cuntz-comparison of open sets if $\Gamma$ has subexponential growth and $(X, \Gamma)$ is an extension of a Cantor system (\cite{Niu-MD-Z}).

In \cite{Niu-MD-Z}, it is shown that if $(X, \Gamma)$ has the (URP) and (COS), then the comparison radius of the C*-algebra $\mathrm{C}(X)\rtimes\Gamma$ is at most half of the mean dimension of $(X, \Gamma)$. In particular, if $\mathrm{mdim}(X,  \Gamma) = 0$, then the C*-algebra $\mathrm{C}(X)\rtimes\Gamma$ has the strict comparison of positive elements (see Definition \ref{rank-fn}), which, as a part of the Toms-Winter conjecture, to imply the $\mathcal Z$-stability (this has been verified in the case that the C*-algebra has finitely many extreme tracial states in \cite{Matui-Sato-CP}, and then been generalized independently to the case that the set of extreme tracial states is finite dimensional in \cite{Sato-CP}, \cite{KR-CenSeq}, and \cite{TWW-Z}, and then to the case that the algebra has Uniform Property Gamma in \cite{CETWW-CP}).

Under the assumption that $(X, \Gamma)$ has the small boundary property (SBP) (which implies zero mean dimension, 
see \cite{Lindenstrauss-Weiss-MD}, and is shown in \cite{Lind-MD} and \cite{GLT-Zk} to be equivalent to zero mean dimension in the case  $\Gamma=\Int^d$), Kerr and Szabo show in \cite{KS-comparison} (Theorem 9.4) that the C*-algebra $\mathrm{C}(X)\rtimes\Gamma$ has the Uniform Property Gamma, and hence the strict comparison of positive elements implies $\mathcal Z$-stability for $\mathrm{C}(X)\rtimes\Gamma$.

In this note, one shows the following:
\theoremstyle{theorem}
\newtheorem*{thmA}{Theorem}
\begin{thmA}[Theorem \ref{Z-stable}]
Let $(X, \Gamma)$ be a free and minimal dynamical system with the (URP) and (COS). If $(X, \Gamma)$ has mean dimension zero, then $(\mathrm{C}(X)\rtimes\Gamma) \otimes \mathcal Z \cong \mathrm{C}(X) \rtimes \Gamma$, where $\mathcal Z$ is the Jiang-Su algebra. 

In particular, let $(X_1, \Gamma_1)$ and $(X_2, \Gamma_2)$ be two free minimal dynamical systems with the (URP) and (COS), and zero mean dimension, then  $$\mathrm{C}(X_1)\rtimes\Gamma_1 \cong \mathrm{C}(X_2)\rtimes\Gamma_2$$ if and only if $$\mathrm{Ell}(\mathrm{C}(X_1)\rtimes\Gamma_1) \cong \mathrm{Ell}(\mathrm{C}(X_2)\rtimes\Gamma_2),$$ where $\mathrm{Ell}(\cdot) =  (\Kzero(\cdot), \Kzero^+(\cdot), [1], \mathrm{T}(\cdot), \rho, \Kone(\cdot))$ is the Elliott invariant. Moreover, these C*-algebras are inductive limits of unital subhomogeneous C*-algebras.
\end{thmA}

As a consequence, the following crossed-product C*-algebras are $\mathcal Z$-stable:
\theoremstyle{theorem}
\newtheorem*{corA}{Corollary}
\begin{corA}[Corollary \ref{Z-corollary}]
Let $(X, \Gamma)$ be a free and minimal dynamical system with mean dimension zero. Assume that
\begin{itemize}
\item either $\Gamma = \Int^d$ for some $d\geq 1$, or
\item $(X, \Gamma)$ is an extension of a Cantor system and $\Gamma$ has subexponetial growth.
\end{itemize}
Then,  the C*-algebra $\mathrm{C}(X) \rtimes \Gamma$ is classified by the Elliott invariant and is an inductive limit of unital subhomogeneous C*-algebras. 
\end{corA}

Two approaches are included in this note: The first approach is more self-contained and more C*-algebra oriented. It is to show that the C*-algebra $\mathrm{C}(X)\rtimes\Gamma$ being considered is tracially $\mathcal Z$-stable; since $\mathrm{C}(X)\rtimes\Gamma$ is nuclear, it follows from \cite{Matui-Sato-CP} and \cite{HO-Z} that $\mathrm{C}(X)\rtimes\Gamma$ actually is $\mathcal Z$-stable.

In the second approach (Section \ref{SBP}), one proves the following dynamical system statement: $$\mathrm{mdim}0 + \mathrm{URP} \Rightarrow \mathrm{SBP}.$$  If, in addition, the system is assumed to have the (COS), it follows from \cite{Niu-MD-Z} that the C*-algebra $\mathrm{C}(X)\rtimes\Gamma$ has strict comparison of positive elements. Hence, with the SBP, the $\mathcal Z$-stability of $\mathrm{C}(X)\rtimes\Gamma$ also follows from the Theorem 9.4 and Corollary 9.5 of \cite{KS-comparison}. 


\section{Notation and preliminaries}

\subsection{Topological Dynamical Systems}

\begin{defn}
A topological dynamical system $(X, \Gamma)$ consists of a separable compact Hausdorff space $X$, a discrete group $\Gamma$, and a homomorphism $\Gamma \to \mathrm{Homeo}(X)$, where $\mathrm{Homeo}(X)$ is the group of homeomorphisms of $X$, acting on $X$ from the right. In this paper, we frequently omit the word topological, and just refer it as a dynamical system. 

The dynamical system $(X, \Gamma)$ is said to be free if $x\gamma=x$ implies $\gamma=e$, where $x\in X$ and $\gamma\in\Gamma$.

A closed set $Y \subseteq X$ is said to be invariant if $$Y\gamma = Y,\quad \gamma \in\Gamma,$$ and the dynamical system $(X, \Gamma)$ is said to be minimal if $\varnothing$ and $X$ are the only invariant closed subsets.

\end{defn}

\begin{defn}
A Borel measure $\mu$ on $X$ is invariant if for any Borel set $E\subseteq X$, one has $$\mu(E) = \mu(E\gamma),\quad \gamma\in\Gamma.$$ Denote by $\mathcal M_1(X, \Gamma)$ the set of all invariant Borel probability measures on $X$. It is a Choquet simplex under the weak* topology.
\end{defn}

\begin{defn}\label{defn-amenable}
Let $\Gamma$ be a (countable) discrete group. Let $K\subseteq\Gamma$ be a finite set and let $\delta>0$. Then a finite set $F \subseteq \Gamma$ is said to be $(K, \eps)$-invariant if $$\frac{\abs{FK\Delta F}}{\abs{F}} < \eps.$$

The group $\Gamma$ is amenable if there is a sequence  $(\Gamma_n)$ of finite subsets of $\Gamma$ such that for any $(K, \eps)$, the set $\Gamma_n$ is $(K, \eps)$-invariant if $n$ is sufficiently large. The sequence $(\Gamma_n)$ is called a F{\o}lner sequence.

The $K$-interior of a finite set $F\subseteq\Gamma$ is defined as $$\mathrm{int}_K(F) = \{\gamma\in F: \gamma K \subseteq F\}.$$ Note that $$\abs{F\setminus\mathrm{int}_K(F)} \leq\abs{K}\abs{FK\setminus F}\leq \abs{K}\abs{FK\Delta F} ,$$ and hence for any $\eps>0$, if $F$ is $(K, \frac{\eps}{\abs{K}})$-invariant, then $$\frac{\abs{F\setminus\mathrm{int}_K(F)}}{\abs{F}} < \eps.$$

\end{defn}

\begin{defn}[see \cite{Lindenstrauss-Weiss-MD}]
Consider a topological dynamical system $(X, \Gamma)$, where $\Gamma$ is amenable, and let $E\subseteq X$. The orbit capacity of $E$ is defined by
$$\mathrm{ocap}(E):=\lim_{n\to\infty}\frac{1}{\abs{\Gamma_n}}\sup_{x\in X}\sum_{\gamma\in \Gamma_n} \chi_E(x\gamma),$$
where $(\Gamma_n)$ is a F{\o}lner sequence, and $\chi_E$ is the characteristic function of $E$.
The limit always exists and is independent from the choice of the F{\o}lner sequence $(\Gamma_n)$.
\end{defn}

\begin{defn}[see \cite{Gromov-MD} and \cite{Lindenstrauss-Weiss-MD}]
Let $\mathcal U$ be an open cover of $X$. Define
$$D(\mathcal U)=\min\{\mathrm{ord}(\mathcal V): \textrm{$\mathcal V$ is an open cover of $X$ and $\mathcal V \preceq U$}\},$$ where $$\mathrm{ord}(\mathcal V)=-1 + \sup_{x\in X}\sum_{V\in\mathcal V} \chi_V(x),$$ and $\mathcal V \preceq \mathcal U$ means that, for any $V\in\mathcal V$, there is $U\in\mathcal U$ with $V\subseteq U$.

Consider a topological dynamical system $(X, \Gamma)$, where $\Gamma$ is a discrete amenable group. The mean topological dimension is defined by
$$\mathrm{mdim}(X, \Gamma):=\sup_{\mathcal U}\lim_{n\to\infty}\frac{1}{\abs{\Gamma_n}}D(\bigvee_{\gamma\in\Gamma_n} \gamma^{-1}(\mathcal U)),$$
where $\mathcal U$ runs over all finite open covers of $X$, $(\Gamma_n)$ is a F{\o}lner sequence (the limit is independent from the choice of $(\Gamma_n)$), and $\alpha \vee \beta$ denotes the open cover $$\{U \cap V: U\in \alpha,\ V\in\beta\}$$ for any open covers $\alpha$ and $\beta$.


\end{defn}


\subsection{Crossed product C*-algebras}
Consider a topological dynamical system $(X, \Gamma)$. 
The (full) crossed product C*-algebra $A=\mathrm{C}(X)\rtimes\Gamma$ is defined to be the universal C*-algebra
$$\textrm{C*}\{f, u_\gamma;\ u_\gamma fu_\gamma^*= f(\cdot \gamma )=f\circ \gamma,\ u_{\gamma_1}u^*_{\gamma_2} = u_{\gamma_1\gamma_2^{-1}},\ u_e=1,\ f\in \mathrm{C}(X),\ \gamma, \gamma_1, \gamma_2 \in \Gamma\}.$$
The C*-algebra $A$ is nuclear (Corollary 7.18 of \cite{Williams}) if $\Gamma$ is amenable. If, moreover, $(X, \Gamma)$ is minimal and topologically free, the C*-algebra $A$ is simple (Theorem 5.16 of \cite{Effros-Hahn} and  Th\'{e}or\`{e}me 5.15 of \cite{ZM-prod}), i.e., $A$ has no non-trivial two-sided ideals.  

\section{The Cuntz semigroup of $\mathrm{C}(X)\rtimes\Gamma$}

\begin{defn}
Let $A$ be a C*-algebra, and let $a, b\in A^+$. The element $a$ is said to be Cuntz sub-equivalent to $b$, denoted by $a \precsim b$, if there are $x_i$, $y_i$, $i=1, 2, ...$, such that $$\lim_{i\to\infty} x_iby_i = a,$$ and we say that $a$ is Cuntz equivalent to $b$, denoted by $a\sim b$, if $a\precsim b$ and $b \precsim a$. Then the Cuntz semigroup of $A$, denoted by $\mathrm{W}(A)$, is defined as
$$ (\mathrm{M}_\infty(A))^+/ \sim$$ with the addition
$$[a] + [b] = \left[\left(\begin{array}{cc} a &  \\ & b\end{array}\right)\right],$$ where $(\mathrm{M}_\infty(A))^+:=\bigcup_{n=1}^\infty\mathrm{M}^+_n(A)$ and  $[\cdot]$ denotes the equivalence class.

\end{defn}

\begin{defn}\label{rank-fn}
Let $A$ be a C*-algebra, let $\mathrm{T}(A)$ denote the set of all tracial states of $A$, equipped with the topology of pointwise convergence. Note that if $A$ is unital, the set $\mathrm T(A)$ is a Choquet simplex.

Let $a$ be a positive element of $\mathrm{M}_\infty(A)$ and $\tau \in \mathrm{T}(A)$; define
$$\mathrm{d}_\tau(a) = \lim_{n\to\infty}\tau(a^{\frac{1}{n}}),$$ where $\tau$ is extended naturally to $\mathrm{M}_\infty(A)$. 
The function $$\mathrm{T}(A) \ni \tau \mapsto \mathrm{d}_\tau(a) \in \Real^+$$ is the limit of an increasing sequence of strictly positive affine functions on $\mathrm{T}(A)$, so it is lower semicontinuous. 

It is well known that if $a \precsim b$, then
$$\mathrm{d}_\tau(a) \leq \mathrm{d}_\tau(b),\quad \tau\in\mathrm{T}(A).$$
If the C*-algebra $A$ satisfies the property that for any positive elements $a, b\in\mathrm{M}_\infty(A)$ with $$\mathrm{d}_\tau(a) < \mathrm{d}_\tau(b),\quad \tau\in\mathrm{T}(A),$$ then $a \precsim b$, the C*-algebra $A$ is said to have the strict comparison of positive elements.
\end{defn}

\begin{rem}
Note that if $A = \mathrm{M}_n(\mathrm{C}_0(X))$, where $X$ is a locally compact Hausdorff space, and $\tau$ be a trace of $A$. Then, for any positive element $a\in \mathrm{M}_\infty(A) \cong \mathrm{M}_\infty(\mathrm{C}_0(X))$ and any $\tau\in\mathrm T(A)$, one has $$\tau(a) =\int_{X} \frac{1}{n}\mathrm{Tr}(a(x))d\mu_\tau \quad\mathrm{and}\quad \mathrm{d}_\tau(a) =\int_{X} \frac{1}{n}\mathrm{rank}(a(x))d\mu_\tau, $$ where $\mu_\tau$ is the Borel measure on $X$ induced by $\tau$.

\end{rem}

\begin{defn}
Let $A$ be a unital C*-algebra. Denote by $\mathrm{LAff}_{\mathrm b}(\mathrm{T}(A))^{++}$ the cone of all strictly positive lower semicontinuous affine functions on $\mathrm{T}(A)$, and denote by $\mathrm{V}(A)$ the semigroup of Murray-von Neumann equivalence classes of projections of $\bigcup_{n=1}^\infty\mathrm{M}_n(A)$. Then $\mathrm{V}(A) \sqcup \mathrm{LAff}_{\mathrm b}(\mathrm{T}(A))^{++}$ form an ordered  semigroup with addition of crossed terms to be $$(p + f) = p(\tau) + f(\tau) \in \mathrm{LAff}_{\mathrm b}(\mathrm{T}(A))^{++},\quad p\in\mathrm{V}(A),\ f\in \mathrm{LAff}_{\mathrm b}(\mathrm{T}(A))^{++}.$$ Then the map
 $$
 \mathrm{W}(A) \ni [a] \mapsto 
 \left\{ 
 \begin{array}{ll} 
 [a] \in \mathrm{V}(A), & \textrm{if $a$ is equivalent to a projection,} \\
 (\tau \mapsto \mathrm{d}_\tau([a])) \in \mathrm{LAff}_{\mathrm b}(\mathrm{T}(A))^{++},& \textrm{otherwise,}  \end{array}
\right.
$$ 
is a representation of the Cuntz semigroup $\mathrm{W}(A)$.

\end{defn}

%

The following is a version of Theorem 3.4 of \cite{Toms-SDG} for the C*-algebra $\mathrm{C}(X)\rtimes\Gamma$.

\begin{prop}\label{cont-approx}
Let $A = \mathrm{C}(X)\rtimes\Gamma$, where $(X, \Gamma)$ is free, minimal,  has the (URP) and zero mean dimension. 
Then, for any continuous affine function $\alpha: \mathrm{T}(A) \to (0, \infty)$ and any $\eps>0$, there is a positive element $a\in \mathrm{M}_\infty(A)$ such that 
$$\abs{\alpha(\tau) - \mathrm{d}_\tau(a)} <\eps,\quad \tau\in\mathrm{T}(A).$$
\end{prop}

\begin{proof}
By Corollary 3.10 of \cite{BPT-Cuntz}, there is a positive element $a'\in A$ such that $$\alpha(\tau) = \tau(a'),\quad \tau\in\mathrm{T}(A).$$ Since the action is minimal, the algebra $A$ is simple, and hence there is $\delta\in (0, 1)$ such that 
\begin{equation}\label{lbd-trace}
\tau(a')>\delta,\quad \tau\in \mathrm{T}(A).
\end{equation}
Also pick $M>\norm{a'}$ so $$\tau(a') < M,\quad \tau\in\mathrm{T}(A).$$

Let $\eps\in (0, \frac{1}{4})$ be arbitrary. By Theorem 3.8 of \cite{Niu-MD-Z}, for any finite subset $\mathcal F\subseteq A$ and any $\eps'\in (0, \eps)$ ($\mathcal F$ and $\eps'$ will be fixed in the next paragraph), there exist $a''\in A$, a finite set $\mathcal F'\subseteq A$,  $h\in \mathrm{C}(X)^+$, and a sub-C*-algebra $C \subseteq A$ with $C\cong\bigoplus_{s=1}^S\mathrm{M}_{n_s}(\mathrm{C}_0(Z_s))$ and closed sets $[Z_s] \subseteq Z_s$  such that
\begin{enumerate}
\item for any $f\in\mathcal F$, there is $f'\in \mathcal F'$ such that $\norm{f - f'} < \eps'$,
\item $\norm{a'-a''}< \eps'$, $\norm{ha''-a''h} < \eps'$, $\norm{hf' - f'h} < \eps'$, $f'\in\mathcal F'$,
\item $h\in C$, $ha''h\in C$, $hf'h\in C$, $f'\in\mathcal F'$,
\item $\norm{h} \leq 1$, $\tau(1-h) < \eps'$, $\tau\in\mathrm{T}_1(A)$, 
\item $\mu(X\setminus h^{-1}(1))<\frac{\eps'}{M+1}$, $\mu\in\mathcal M_1(X, \Gamma)$,
\item\label{lbd-cut-tr} under the isomorphism $C\cong \bigoplus_{s=1}^S\mathrm{M}_{n_s}(\mathrm{C}_0(Z_s))$, the element $h$ has the form $$h = \bigoplus_{s=1}^S\mathrm{diag}\{h_{s, 1}, ..., h_{s, n_s}\},$$ where $h_{s, i}: Z_s \to [0, 1]$, and
$$\frac{1}{n_s}\abs{\{1\leq i\leq n_s: h_{s, i}(x) = 1 \}} >1-\eps', \quad x\in [Z_s],\ s=1, ..., S,$$ 
\item $$\frac{\mathrm{dim}([Z_s])}{n_s} < \eps'\delta,\quad s=1, 2,..., S,$$
\item\label{n-large} each $n_s$ is sufficiently large such that the interval $(2n_s\delta\eps+1, 4n_s \eps-1)$ contains at least one integer.
\end{enumerate}

Put $$a'_0=(1-h)^{\frac{1}{2}} a'' (1-h)^{\frac{1}{2}}\quad\mathrm{and}\quad a'_1=h^{\frac{1}{2}} a'' h^{\frac{1}{2}}.$$ One asserts that with $\mathcal F$ sufficiently large and $\eps'$ sufficiently small, one has
\begin{equation}\label{lbd-cut}
M> \tau(\pi(a_1')) > \delta,\quad \tau\in\mathrm{T}(\pi(C)),
\end{equation}
where $\pi$ is the standard quotient map from $C\cong\bigoplus_{s=1}^S\mathrm{M}_n(\mathrm{C}_0(Z_s))$ to $\bigoplus_{s=1}^S\mathrm{M}_n(\mathrm{C}_0([Z_s]))$.
Then, fix this pair of $(\mathcal F, \eps')$.

Indeed, suppose the contrary, there then exist a sequence of finite subset $\mathcal F'_i\subseteq A$ with dense union and a sequence of positive numbers $\eps_i$ decreasing to $0$, sub-C*-algebras $C_i\subseteq A$, $a''_i\in A$, positive elements $h_i\in C_i$ such that
\begin{itemize}
\item $\norm{a'-a''_i} < \eps_i$,
\item $\norm{h^{\frac{1}{4}}_if' -f' h_i^{\frac{1}{4}}} <\eps_i$, $f'\in\mathcal F'_i$,
\item $h_ia_i''h_i \in C_i$, $h_i\in C_i$, and $h_if'h_i\in C_i$, $f'\in \mathcal F'_i$, so that  $$h_i^{\frac{1}{2}} a''_i h_i^{\frac{1}{2}} \in C_i,\quad h_i^{\frac{1}{4}} a''_i h_i^{\frac{1}{4}} \in C_i,\quad h_i^{\frac{1}{2}} f' h_i^{\frac{1}{2}} \in C_i,\quad\mathrm{and}\quad h_i^{\frac{1}{4}} f' h_i^{\frac{1}{4}} \in C_i,\quad f'\in\mathcal F'_i,$$
\item there exists $\tau_i\in\mathrm{T}(\pi(C_i))$ such that $$ \tau_i(\pi_i(h^{\frac{1}{2}} a''_i h_i^{\frac{1}{2}}))\leq\delta\quad\mathrm{or}\quad \tau_i(\pi_i(h^{\frac{1}{2}} a''_i h_i^{\frac{1}{2}}))\geq M,$$ where $\pi_i$ is the standard quotient map from $C_i\cong\bigoplus_{s=1}^S\mathrm{M}_n(\mathrm{C}_0(Z_s))$ to $\bigoplus_{s=1}^S\mathrm{M}_n(\mathrm{C}_0([Z_s]))$,
\item $\tau(\pi(h_i)) > 1-\eps_i$, for any $\tau\in\mathrm{T}(\pi(C_i))$ (this follows from \ref{lbd-cut-tr}).
\end{itemize}

Consider the linear functional
$$\rho_i: A \ni a \mapsto \tau_i(\pi(h_i^{\frac{1}{2}} a h_i^{\frac{1}{2}})) \in \Comp,$$ and note that 
$$\norm{\rho_i} = \rho_i(1_A) = \tau_i(\pi(h_i)) > 1- \eps_i .$$
Also note that for, any $a, b\in\mathcal F'_i$, 
\begin{eqnarray*}
\rho_i(ab) & = & \tau_i(\pi(h_i^{\frac{1}{2}} ab h_i^{\frac{1}{2}})) \approx_{2\eps_i} \tau_i(\pi(h_i^{\frac{1}{4}} a h_i^{\frac{1}{4}} h_i^{\frac{1}{4}} b h_i^{\frac{1}{4}})) =  \tau_i(\pi(h_i^{\frac{1}{4}} b h_i^{\frac{1}{4}} h_i^{\frac{1}{4}} a h_i^{\frac{1}{4}})) \\ 
&\approx_{2\eps_i}&  \tau_i(\pi(h_i^{\frac{1}{2}} ba h_i^{\frac{1}{2}})) = \rho_i(ba).
\end{eqnarray*}
Thus, any accumulation point of $\{\rho_i\}$, say $\rho_\infty$, is actually a tracial state. However, $$\rho_\infty(a') = \lim_{i\to\infty} \tau_i(\pi_i(h^{\frac{1}{2}} a''_i h_i^{\frac{1}{2}})) \leq \delta\quad\mathrm{or}\quad \rho_\infty(a') = \lim_{i\to\infty} \tau_i(\pi_i(h^{\frac{1}{2}} a''_i h_i^{\frac{1}{2}})) \geq M,$$ which contradicts to \eqref{lbd-trace}. This proves the assertion.

Denote by $Z$ the (abstract) disjoint union of $Z_s$, $s=1, ..., S$, and denote by $[Z]$ the (abstract) disjoint union of $[Z_s]$, $s=1, ..., S$.
Consider $\pi(a_1') \in \pi(C)$, and consider the continuous function $$[Z] \ni x\mapsto \mathrm{Tr}(\pi(a_1')(x)) \in (0, +\infty).$$

For each $s=1, 2, ..., S$, by \ref{n-large}, one picks an integer $$\Delta_s \in (2n_s\delta\eps+1, 4n_s \eps-1).$$
Define
$$f: [Z] \ni x \mapsto \lceil \mathrm{Tr}(\pi(a_1')(x))\rceil + \Delta_s,\quad \mathrm{if}\ x\in [Z_s] $$ and
$$g: [Z] \ni x \mapsto \lfloor \mathrm{Tr}(\pi(a_1')(x))\rfloor - \Delta_s, \quad \mathrm{if}\ x\in [Z_s] $$
where $\lfloor t \rfloor = \max\{k\in\Int: k\leq t\}$ and  $\lceil t \rceil=\min\{k\in\Int: k\geq t\}$.
Note that by \eqref{lbd-cut}, for any $x\in [Z_s]$, $s=1, ..., S$, one has
$$\lfloor \mathrm{Tr}(\pi(a_1')(x))\rfloor - \Delta_s \geq n_s\mathrm{tr}(\pi(a_1')(x)) -2n_s\delta\eps-1 > n_s\delta -2n_s\delta\eps-1 >0.$$ That is, the function $g$ is a positive. Also note that for any $x\in [Z_s]$, $s=1, ..., S$, 
\begin{eqnarray*}
f(x)&\leq  & \max\{ \lceil\mathrm{Tr}(\pi(a_1')(y))\rceil + \Delta_s: y\in[Z_s] \} \\
&\leq & \max\{\mathrm{Tr}(\pi(a_1')(y)) + 4 n_s\eps: y\in[Z_s] \} \\
&\leq & n_s\max\{\mathrm{tr}(\pi(a_1')(y)) + 4\eps: y\in[Z_s] \}\\
&\leq & n_s(M+1).
\end{eqnarray*}
Therefore $f$ and $g$ satisfy
\begin{enumerate}
\item[(a)] $g$ is positive upper semicontinuous and $f$ is lower semicontinuous,
\item[(b)] $0< g(x) < \mathrm{Tr}(\pi(a_1')(x))< f(x) \leq n_s(M+1)$, $x\in [Z_s]$, and
\item[(c)] $4\mathrm{dim}([Z_s]) < 4n_s \delta\eps < 2\Delta_s-2 < f(x) - g(x) \leq 2\Delta_s+2 < 8 \eps n_s$, $x\in [Z_s]$.
\end{enumerate}

It then follows from Proposition 2.9 of \cite{Toms-SDG} that there is a positive element $a''' \in \mathrm{M}_\infty(\pi(C))$ such that 
$$g(x) < \mathrm{rank}(a'''(x))< f(x),\quad x\in [Z_s].$$

Extend $a'''$ to an element of $\mathrm{M}_n(C)\subseteq \mathrm{M}_n(A)$ and denote it by $a$. One then has that for any $x\in[Z]$, 
\begin{eqnarray}
&&\abs{\frac{1}{n(x)}\mathrm{rank}(a(x)) - \mathrm{tr}(a_1'(x))} \label{appro-rank} \\
& = & \abs{\frac{1}{n(x)}\mathrm{rank}(a'''(x)) - \mathrm{tr}(a_1'(x))} \nonumber \\
& = & \abs{\frac{1}{n(x)}\mathrm{Tr}(a_1'(x)) - \mathrm{tr}(a_1'(x))} + \frac{1}{n(x)}(f(x) - g(x)) \nonumber \\
& \leq & 8\eps. \nonumber 
\end{eqnarray}
Note that the element $a$ can be chosen so that for any $x\in Z_s\setminus[Z_s]$, $s=1, ..., S$, 
\begin{eqnarray*}
\mathrm{rank}(a(x)) & \leq & \max\{f(x): x\in[Z_s]\}\leq n_s(M+1).
\end{eqnarray*}

Now, let $\tau\in\mathrm{T}(A)$ be arbitray, and let $\mu_\tau$ denote the Borel measure on $Z$ induced by the restriction of $\tau$ to $C$. Note that $1-\eps < \norm{\mu_\tau} \leq 1$ (since $\tau(h)\geq 1-\eps'>1-\eps$), and also note that
$$\mu_\tau(Z\setminus[Z]) \leq \mathrm{d}_\tau(\tilde{c} - h) \leq \mathrm{d}_\tau(1_A - h) < \mu(X\setminus h^{-1}(1)) <\frac{\eps}{M+1},$$
where $\tilde{c} \geq h$ is some strict positive element of $C\subseteq A$, and $\mu$ is the invariant measure on $X$ corresponding to $\tau$ ($\mu$ is not $\mu_\tau$). Therefore, 
$$\int_{Z\setminus[Z]} \frac{1}{n(x)} \mathrm{rank}(a(x)) d\mu_\tau \leq \int_{Z\setminus[Z]} (M+1) d\mu_\tau< \eps,$$
and (by \ref{lbd-cut})
$$\int_{Z\setminus[Z]} \mathrm{tr}(a_1'(x)) d\mu_\tau \leq \int_{Z\setminus[Z]} \norm{a'_1} d\mu_\tau \leq \int_{Z\setminus[Z]} \norm{a_1} d\mu_\tau \leq   \int_{Z\setminus[Z]} M d\mu_\tau < \eps,$$ where $n(x) = n_s$ if $x\in Z_s$.
In particular
$$\abs{\int_{Z\setminus[Z]} \frac{1}{n(x)} \mathrm{rank}(a(x)) d\mu_\tau - \int_{Z\setminus[Z]} \mathrm{tr}(a_1'(x)) d\mu_\tau} < 2\eps.$$
By \eqref{appro-rank},
\begin{eqnarray*}
\mathrm{d}_\tau(a) & = & \int_Z \frac{1}{n(x)} \mathrm{rank}(a(x)) d\mu_\tau \\ 
& = &  \int_{[Z]} \frac{1}{n(x)} \mathrm{rank}(a(x)) d\mu_\tau +  \int_{Z\setminus[Z]} \frac{1}{n(x)} \mathrm{rank}(a(x)) d\mu_\tau \\
&\approx_{2\eps} & \int_{[Z]} \frac{1}{n(x)} \mathrm{rank}(a(x)) d\mu_\tau +  \int_{Z\setminus[Z]} \mathrm{tr}(a_1'(x)) d\mu_\tau \\
&\approx_{8\eps} & \int_{[Z]} \mathrm{tr}(a_1'(x)) d\mu_\tau +  \int_{Z\setminus[Z]} \mathrm{tr}(a_1'(x)) d\mu_\tau \\
&= &  \int_{Z} \mathrm{tr}(a_1'(x)) d\mu_\tau = \tau(a_1') \\
&\approx_{2\eps} & \tau(a').
\end{eqnarray*}
Since $\eps$ is arbitrary, this proves the desired conclusion.
\end{proof}

\begin{cor}\label{md0-divisible}
Let $(X, \Gamma)$ be a free and minimal dynamical system with the (URP), and assume that $A=\mathrm{C}(X)\rtimes\Gamma$ has Cuntz comparison of open sets. If $(X, \Gamma)$ has mean dimension zero, then, \begin{equation}\label{Cuntz-Aff}
\mathrm{W}(A) \cong \mathrm{V}(A)\sqcup \mathrm{LAff}_{\mathrm b}(\mathrm{T}(A))^{++}.
\end{equation} 
In particular, for any positive element $a\in\mathrm{M}_\infty(A)$, and any $k\in\mathbb N$, 
there is $x\in\mathrm{W}(A)$ such that 
\begin{equation}\label{div}
kx \leq [a] \leq (k+1) x.
\end{equation} 
In other words, $A$ is $0$-almost divisible (hence tracially $0$-divisible)(see Definition 3.5(i) of \cite{Winter-Z-stable-02}).
\end{cor}
\begin{proof}
By Theorem 4.8 of \cite{Niu-MD-Z}, the C*-algebra $A$ has strict comparison of positive elements. Then \eqref{Cuntz-Aff} follows from Proposition \ref{cont-approx} and the proof of Theorem 5.3 of \cite{BPT-Cuntz}.

Let $a\in \mathrm{M}_\infty(A)$ be a non-zero positive element, and pick $\delta>0$ such that $$\mathrm{d}_\tau(a) > \delta,\quad \tau\in\mathrm{T}(A).$$ Since $A$ is simple and non-elementary, there is a non-zero positive element $c\in A$ such that $\mathrm{sp}(c) = [0, 1]$ and $$\mathrm{d}_\tau(c) < \frac{\delta}{k}< \frac{1}{k}\mathrm{d}_\tau(a),\quad \tau\in\mathrm{T}(A).$$

Consider $[a] + [c]$, and note that $[a] + [c] \in \mathrm{LAff}_{\mathrm b}(\mathrm{T}(A))^{++}$. By \eqref{Cuntz-Aff}, there is $x\in\mathrm{M}_\infty(A)$, which is not Cuntz equivalent to a projection, such that
$$\mathrm{d}_\tau(x) = \frac{1}{k+1} \mathrm{d}_\tau([a]+[c]),\quad \tau \in\mathrm{T}(A).$$

Then, for any $\tau \in \mathrm{T}(A)$, 
\begin{eqnarray*}
k(\mathrm{d}_\tau(x)) & = & \frac{k}{k+1} \mathrm{d}_\tau([a]+[c])\\ 
& = & \frac{k}{k+1} \mathrm{d}_\tau([a])+ \frac{k}{k+1}\mathrm{d}_\tau(c) \\ 
& < & \frac{k}{k+1} \mathrm{d}_\tau(a)+ \frac{1}{k+1} \mathrm{d}_\tau(a) = \mathrm{d}_\tau(a), 
\end{eqnarray*}
and
$$(k+1)(\mathrm{d}_\tau(x)) = \mathrm{d}([a]+[c])(\tau) > \mathrm{d}_\tau(a).$$ Together with \eqref{Cuntz-Aff}, this proves \eqref{div}.
\end{proof}


\begin{rem}
Note that a straightforward argument shows that there is $m$ such that for any $k\in\mathbb N$, there is $x\in \mathrm{W}(A)$ such that  $$kx \leq [1_A] \leq m(k+1) x,$$ whenever $(X, \Gamma)$ has the (URP) and Cuntz-comparison of open sets, even without mean dimension zero. Then, as a natural question, is $\mathrm{C}(X)\rtimes\Gamma$ always tracially $m$-divisible for some $m\in\mathbb N$ if $(X, \Gamma)$ has the (URP) and Cuntz-comparison of open sets, but without any assumptions on mean dimension?
\end{rem}

\section{Approximate central order zero maps from $\mathrm{M}_k(\Comp)$ to $\mathrm{C}(X)\rtimes\Gamma$ and the $\mathcal Z$-stability of $\mathrm{C}(X)\rtimes\Gamma$} 

One considers the $\mathcal Z$-stability of $\mathrm{C}(X)\rtimes\Gamma$ in this section. First, one has the following lemma which is enssentailly Theorem 3.8 of \cite{Niu-MD-Z}, stating that the C*-algebra $A=\mathrm{C}(X)\rtimes\Gamma$ can be (weakly) tracially approximated by homogeneous C*-algebras, but with an extra conclusion that there is an element $h$ in the homogeneous sub-C*-algebra, which is approximately central in $A$, large in trace, and is orthogonal to the elements with smaller trace in the decomposition obtained from the tracial approximation.
\begin{lem}\label{cut-down}
Let $(X, \Gamma)$ be a free dynamical system with the (URP). Then, for any finitely many elements $f_1, f_2, ..., f_n\in \mathrm{C}(X)\rtimes\Gamma$ and any $\eps>0$, there exist a C*-algebra $C\subseteq \mathrm{C}(X)\rtimes\Gamma$ with $C \cong \bigoplus_{s=1}^S \mathrm{M}_{k_s}(\mathrm{C}_0(U_s))$ for some $k_s\in\mathbb N$ and locally compact Hausdorff space $U_s$, $s=1, ..., S$,  positive functions $h \in\mathrm{C}(X) \cap C$, and $f_1^{(0)}, f_1^{(1)}, f_2^{(0)}, f_2^{(1)}, ... , f_n^{(0)}, f_n^{(1)}\in \mathrm{C}(X)\rtimes\Gamma$ such that

\begin{enumerate}
\item $\norm{f_i - (f_i^{(0)} + f_i^{(1)})} < \eps$, $1\leq i\leq n$,
\item $f_i^{(1)} \in C$, $1\leq i\leq n$,
\item $\norm{f_i^{(0)} h} =0 $, $1\leq i\leq n$, 
\item $\norm{[f_i^{(1)}, h]} < \eps$, $1\leq i\leq n$, 
\item $\tau(1-h^2) < \eps$, $\tau\in\mathrm{T}(A)$.
\end{enumerate}
\end{lem}

\begin{proof}
The proof is similar to that of Theorem 3.8 of \cite{Niu-MD-Z}, but without dealing with mean dimension.

Denote by $A$ the crossed product C*-algebra $\mathrm{C}(X)\rtimes\Gamma$. Without loss of generality, one may assume 
$$f_i=\sum_{\gamma \in\mathcal N} f_{i, \gamma}u_\gamma$$
for some finite set $\mathcal N \subseteq \Gamma$ with $e\in \mathcal N = \mathcal N^{-1}$, and some $f_{i, \gamma} \in \mathrm{C}(X)$. Denote by 
$$M=\max\{1, \norm{f_{i, \gamma}}: i=1, ..., n, \gamma\in\mathcal N\}.$$

For the given $\eps>0$, choose $\eps_1\in (0, \eps)$ such that if a positive element $a\in A$ with $\norm{a} \leq 1$ satisfies 
$$ \norm{af_i - f_ia} < \eps_1,\quad 1\leq i\leq n,$$ then $$ \norm{a^\frac{1}{2}f_i - f_ia^\frac{1}{2}} < \frac{\eps}{2},\quad 1\leq i\leq n.$$

Pick a natural number $$L > \frac{M\abs{\mathcal N}}{\eps_1},$$ and pick a sufficiently large finite set $K\subseteq\Gamma$ and a sufficiently small  positive number $\delta$ so that if a finite set $\Gamma_0\subseteq\Gamma$ is $(K, \delta)$-invariant, then
\begin{equation}\label{very-small-bd}
\frac{\abs{\Gamma_0 \setminus \mathrm{int}_{\mathcal N^{L+1}} (\Gamma_0)} }{\abs{\Gamma_0}} < \frac{\eps}{2}.
\end{equation}

Since $(X, \Gamma)$ has the (URP), there exist closed sets $B_1, B_2, ..., B_S \subset X$ and $(K, \delta)$-invariant sets $\Gamma_1, \Gamma_2, ..., \Gamma_S \subseteq \Gamma$ such that
$$B_s\gamma,\quad \gamma\in \Gamma_s,\ s=1, ..., S, $$
are mutually disjoint and
$$\mathrm{ocap}(X\setminus\bigsqcup_{s=1}^S\bigsqcup_{\gamma\in \Gamma_s} B_s\gamma) < \frac{\eps}{2}.$$

Pick two open sets $U_s, V_s\subseteq X$, $s=1, 2, ..., S$, satisfying $$U_s \supseteq V_s \supseteq B_s,\quad  U_s \supseteq \overline{V_s},$$ and
$$U_s\gamma,\quad \gamma\in\Gamma_s,\ s=1, ..., S,$$
are mutually disjoint.

Consider the sub-C*-algebra
\begin{equation}\label{defn-C}
C:=\mathrm{C}^*\{ u^*_\gamma f: f\in\mathrm{C}_0(U_s), \gamma\in \Gamma_s, s=1, 2, ..., S\}\subseteq \mathrm{C}(X) \rtimes \Gamma,
\end{equation} 
which, by Lemma 3.11 of \cite{Niu-MD-Z}, is isomorphic to $$\bigoplus_{s=1}^S\mathrm{M}_{\abs{\Gamma_s}}(\mathrm{C}_0(U_s)).$$

For each $s=1, 2, ..., S$, pick continuous functions $\chi_{U_s}, \chi_{V_s}: X \to [0, 1]$ such that 
\begin{equation}\label{chi-cond}
\chi_{U_s}|_{V_s} = 1, \quad \chi_{V_s}|_{B_s} = 1,\quad \chi_{U_s}|_{X\setminus U_s} = 0,\quad\mathrm{and}\quad \chi_{V_s}|_{X\setminus V_s} = 0.
\end{equation}
Note that $\chi_{U_s}, \chi_{V_s} \in C$, and 
\begin{equation}\label{cut-in-C} 
\chi_{U_s} f,\  \chi_{V_s} f \in C,\quad f\in\mathrm{C}(X).
\end{equation}

For each $\Gamma_s$, $s=1, 2, ..., S$, define the subsets 
$$
\left\{
\begin{array}{lll}
\Gamma_{s, L+1} & = & \mathrm{int}_{\mathcal N^{L+1}} (\Gamma_s), \\
\Gamma_{s, L} & = & \mathrm{int}_{\mathcal N^{L}} (\Gamma_s) \setminus \mathrm{int}_{\mathcal N^{L+1}} (\Gamma_s), \\
\Gamma_{s, L-1} & = & \mathrm{int}_{\mathcal N^{L-1}} (\Gamma_s) \setminus \mathrm{int}_{\mathcal N^{L}} (\Gamma_s), \\
\vdots & \vdots & \vdots \\
\Gamma_{s, 0} & = & \Gamma_s \setminus \mathrm{int}_{\mathcal N} (\Gamma_s).
\end{array}
\right.
$$

%
Then, for any $\gamma \in\mathcal N$, one has 
\begin{equation}\label{action-nbhd}
\Gamma_{s, l}\gamma \subseteq  \Gamma_{s, l-1} \cup \Gamma_{s, l} \cup \Gamma_{s, l+1},\quad 1\leq l \leq L.  \end{equation}

Indeed,  pick an arbitrary $\gamma'\in \Gamma_{s, l}$. By the construction, one has 
\begin{equation}\label{eq-contain} 
\gamma' \mathcal N^{l} \subseteq \Gamma_s\quad\mathrm{but}\quad \gamma' \mathcal N^{l+1} \nsubseteq \Gamma_s.
\end{equation} 
Therefore $$ \gamma'\gamma \mathcal N^{l-1}  \subseteq \gamma' \mathcal N^l \subseteq \Gamma_s$$ and hence $\gamma'\gamma \in \mathrm{int}_{\mathcal N^{l-1}}\Gamma_s$ (since $e\in\mathcal N^{l-1}$). 

Thus, to show \eqref{action-nbhd}, one only has to show that $\gamma'\gamma \notin \mathrm{int}_{\mathcal N^{l+2}}\Gamma_s$. Suppose $\gamma'\gamma\mathcal N^{l+2} \subseteq \Gamma_s$. Since $\mathcal N$ is symmetric, one has $\gamma^{-1}\in \mathcal N$; hence $\mathcal N^{l+1} \subseteq \gamma \mathcal N^{l+2}$ and 
$$\gamma' \mathcal N^{l+1} \subseteq \gamma'\gamma \mathcal N^{l+2} \subseteq \Gamma_s,$$ which contradicts \eqref{eq-contain}.

Also note that
\begin{equation}\label{action-nbhd-L-end}
\Gamma_{s, L+1}\gamma \subseteq  \Gamma_{s, L+1} \cup \Gamma_{s, L}.
\end{equation}

For each $\gamma \in \Gamma_s$, define
$$\ell(\gamma) = l,\quad\textrm{if $\gamma \in \Gamma_{s, l}$}.$$ By \eqref{action-nbhd} and \eqref{action-nbhd-L-end}, the function $\ell$ satisfies 
\begin{equation}\label{action-nbhd-fn}
\abs{\ell(\gamma'\gamma) - \ell(\gamma)} \leq 1,\quad \gamma'\in \mathcal N,\ \gamma\in \Gamma_{s, 1}\cup\cdots\cup\Gamma_{s, L+1}.
\end{equation}

Define
$$h_{U} = \sum_{s=1}^S\sum_{l=1}^{L+1} \sum_{\gamma\in \Gamma_{s, l}} \frac{l-1}{L}(\chi_{U_s}\circ \gamma^{-1}) = \sum_{s=1}^S \sum_{l=1}^{L+1} \sum_{\gamma\in \Gamma_{s, l}} \frac{l-1}{L} u^*_\gamma\chi_{U_s} u_\gamma \in\mathrm{C}(X)\cap C,$$
and
$$h_{V} = \sum_{s=1}^S\sum_{l=1}^{L+1} \sum_{\gamma\in \Gamma_{s, l}} \frac{l-1}{L}(\chi_{V_s}\circ \gamma^{-1}) = \sum_{s=1}^S \sum_{l=1}^{L+1} \sum_{\gamma\in \Gamma_{s, l}} \frac{l-1}{L} u^*_\gamma\chi_{V_s} u_\gamma \in\mathrm{C}(X)\cap C.$$

Note that
\begin{eqnarray*}
h_{U} h_{V} & = & ( \sum_{s=1}^S \sum_{l=1}^{L+1} \sum_{\gamma\in \Gamma_{s, l}} \frac{l-1}{L} u^*_\gamma\chi_{U_s} u_\gamma ) ( \sum_{s=1}^S \sum_{l=1}^{L+1} \sum_{\gamma\in \Gamma_{s, l}} \frac{l-1}{L} u^*_\gamma\chi_{V_s} u_\gamma ) \\
& = &  \sum_{s=1}^S \sum_{l=1}^{L+1} \sum_{\gamma\in \Gamma_{s, l}} \frac{l-1}{L} u^*_\gamma\chi_{U_s} \chi_{V_s} u_\gamma  \\
& = &  \sum_{s=1}^S \sum_{l=1}^{L+1} \sum_{\gamma\in \Gamma_{s, l}} \frac{l-1}{L} u^*_\gamma \chi_{V_s} u_\gamma  = h_{V},
\end{eqnarray*}
and hence
\begin{equation}\label{perp-cut}
(1- h_{U}) h_{V} = 0.
\end{equation}

By \eqref{chi-cond} (and \eqref{very-small-bd}), 
\begin{eqnarray*}
 \mathrm{ocap}(X \setminus h_{V}^{-1}(1)) & \leq & \max\{\frac{\abs{\Gamma_s \setminus \mathrm{int}_{\mathcal N^{L+1}} (\Gamma_s)} }{\abs{\Gamma_s}} : s=1, ..., S\}  \\ 
&& + \mathrm{ocap}(X\setminus\bigsqcup_{s=1}^S\bigsqcup_{\gamma\in \Gamma_s} B_s\gamma)  \leq \frac{\eps}{2} + \frac{\eps}{2}\\
& < & \eps,
\end{eqnarray*}
and therefore
$$\mathrm{\tau}(1-h_{V}^2) < \eps,\quad \tau \in \mathrm{T}(A).$$

Note that, by the construction of $C$ (see \eqref{defn-C}), $$\chi_{U_s}^{\frac{1}{2}} u_\gamma \in C,\quad \gamma \in\Gamma_s.$$
Hence, for each $\gamma'\in \mathcal N$, since $\gamma\gamma' \in \Gamma_s$, $\gamma\in \Gamma_{s, l}$, $l=1, 2, ..., L+1$, one has
$$ h_{U}u_{\gamma'} = \sum_{s=1}^S \sum_{l=1}^{L+1} \sum_{\gamma\in \Gamma_{s, l}} \frac{l-1}{L} u^*_\gamma\chi_{U_s} u_{\gamma\gamma'}  = \sum_{s=1}^S \sum_{l=1}^{L+1} \sum_{\gamma\in \Gamma_{s, l}} \frac{l-1}{L} (u^*_\gamma\chi^{\frac{1}{2}}_{U_s}) (\chi^{\frac{1}{2}}_{U_s} u_{\gamma\gamma'})  \in C,$$
and therefore, $$h_{U}u_\gamma h_{U}\in C,\quad \gamma \in \mathcal N.$$ 
For any $f\in \mathrm{C}(X)$, by \eqref{cut-in-C}, one has
$$h_{U}f = \sum_{s=1}^S \sum_{l=1}^{L+1} \sum_{\gamma\in \Gamma_{s, l}} \frac{l-1}{L} u^*_\gamma\chi_{U_s} u_\gamma f =  \sum_{s=1}^S \sum_{l=1}^{L+1} \sum_{\gamma\in \Gamma_{s, l}} \frac{l-1}{L} u^*_\gamma\chi_{U_s} (u_\gamma f u^*_\gamma)u_\gamma \in C,$$ and therefore
$$h_{U}f_i h_{U}\in C,\quad 1 \leq i\leq n.$$

Note that, for each $\gamma' \in\mathcal N$, by \eqref{action-nbhd-fn}, 
\begin{eqnarray*}
&& \norm{u^*_{\gamma'}h_{U}u_{\gamma'} - h_{U}}  \\
& = & \norm{ \sum_{s=1}^S\sum_{l=1}^{L+1} \sum_{\gamma\in \Gamma_{s, l}} \frac{l-1}{L}\chi_{U_s}\circ (\gamma'\gamma)^{-1}  -  \sum_{s=1}^S\sum_{l=1}^{L+1} \sum_{\gamma\in \Gamma_{s, l}} \frac{l-1}{L} \chi_{U_s}\circ \gamma^{-1}} \\
& = & \max\{ \abs{ \frac{\ell(\gamma'\gamma)-1}{L} - \frac{\ell(\gamma)-1}{L}}: \gamma\in\Gamma_s\setminus\Gamma_{s, 0},\ s=1, 2, ..., S \}\\
& < & \frac{1}{L}  <  \frac{\eps_1}{M\abs{\mathcal{N}}},
\end{eqnarray*}
and hence
\begin{equation}\label{pre-comm-p}
\norm{h_{U}f_i - f_i h_{U}} < {\eps_1},\quad i=1, 2, ..., n.
\end{equation}
The same argument also shows that
\begin{equation}\label{pre-comm-p-1}
\norm{h_{V}f_i - f_i h_{V}} < \eps_1<\eps,\quad i=1, 2, ..., n.
\end{equation}

It follows from \eqref{pre-comm-p} and the choice of $\eps_1$ that 
$$\norm{h^{\frac{1}{2}}_{U}f_i - f_i h^{\frac{1}{2}}_{U}} < \frac{\eps}{2}\quad\mathrm{and}\quad \norm{(1-h_{U})^{\frac{1}{2}}f_i - f_i (1-h_{U})^{\frac{1}{2}}} < \frac{\eps}{2},\quad i=1, 2, ..., n,$$ and hence
$$\norm{f_i - ((1-h_{U})^{\frac{1}{2}} f_i (1-h_{U})^{\frac{1}{2}} + h_{U}^{\frac{1}{2}} f_i h_{U}^{\frac{1}{2}}) } < \eps,\quad 1\leq i\leq n. $$

Put
$$f_i^{(0)} = (1-h_{U})^{\frac{1}{2}} f_i (1-h_{U})^{\frac{1}{2}}\quad\mathrm{and}\quad f_i^{(1)} = h_{U}^{\frac{1}{2}} f_i h_{U}^{\frac{1}{2}}.$$
By \eqref{perp-cut},
$$f_i^{(0)} h_{V} = 0,\quad i=1, ..., n.$$
One also has, by \eqref{pre-comm-p-1},
$$f_i^{(1)} h_{V} = h_{U}^{\frac{1}{2}} f_i h_{U}^{\frac{1}{2}} h_{V} = h_{U}^{\frac{1}{2}} f_i h_{V} h_{U}^{\frac{1}{2}}   \approx_{\eps}  h_{U}^{\frac{1}{2}}  h_{V} f_i  h_{U}^{\frac{1}{2}}  = h_{V} h_{U}^{\frac{1}{2}}  f_i  h_{U}^{\frac{1}{2}} =  h_{V} f_i^{(1)}. $$ Thus $$\norm{f_i^{(1)}h_{V} - h_{V}f_i^{(1)}} < \eps,\quad i=1, ..., n.$$ Then the element $h:=h_{V}$ satisfies the lemma.
\end{proof}

Recall
\begin{defn}[\cite{WZ-ndim}]
Let $A$, $B$ be C*-algebras, and let $\varphi: A \to B$ be a completely positive contractive linear map (c.p.c~map). $\varphi$ is said to be order zero if $$a \perp b \Longrightarrow \varphi(a) \perp \varphi(b),\quad a, b\in A.$$
\end{defn}

\begin{defn}[\cite{HO-Z}]
A unital C*-algebra $A$ is said to be tracially $\mathcal Z$-stable if for any finite set $\mathcal F \subseteq A$, any $\eps>0$, and any non-zero positive element $a\in A$, there is a c.p.c.~order zero map $\varphi: \mathrm{M}_2(\Comp) \to A$ such that 
\begin{enumerate}
\item $\norm{[\varphi(a), f]} < \eps$, $a\in\mathrm{M}_2(\Comp)$, $\norm{a} \leq 1$, $f\in\mathcal F$,
\item $1_A-\varphi(1_2) \precsim a$.
\end{enumerate}
\end{defn}

Based on \cite{Matui-Sato-CP}, for nuclear C*-algebras, the tracial $\mathcal Z$-stability is shown to be equivalent to the $\mathcal Z$-stability in \cite{HO-Z}:
\begin{thm}\label{Z-stable}
Let $A$ be a simple separable unital nuclear C*-algebra. Then $A\cong A\otimes \mathcal Z$ if and only if $A$ is tracially $\mathcal Z$-stable, where $\mathcal Z$ is the Jiang-Su algebra.
\end{thm}

\begin{rem}
In general, there are non-nuclear C*-algebras which are tracially $\mathcal Z$-stable but not $\mathcal Z$-stable (see \cite{NW-TAF}).
\end{rem}

The following two lemmas are simple observations.
\begin{lem}\label{L-trace}
Let $A$ be a unital C*-algebra, and let $\tau$ be a tracial state of $A$. Assume $a, b\in A$ are positive elements with norm at most $1$ and $$\tau(1-a)<\eps\quad\mathrm{and}\quad\tau(1-b)<\eps,$$ then
$$\tau(ab) > 1-2\eps. $$
\end{lem}
\begin{proof}
It follows from the assumption that
$$1-\eps<\tau(a)\quad\mathrm{and}\quad -\eps <\tau(b-1).$$
Also note that $$0 \leq \tau((1-a)^{\frac{1}{2}}(1-b)(1-a)^{\frac{1}{2}}) = \tau((1-a)(1-b)) = \tau(1-a-b+ab),$$ and so
$$\tau(a+b-1) \leq \tau(ab).$$ Then
$$1-2\eps=(1-\eps) -\eps<\tau(a) + \tau(b-1) =  \tau(a+b-1) \leq \tau(ab),$$
as desired.
\end{proof}

\begin{lem}\label{pert}
For any $\eps>0$, if $\varphi: \mathrm{M}_k(\Comp) \to  A$ is a c.p.c.~order zero map with
$$\tau(1_A-\phi(1_k)) < \eps,\quad \tau\in\mathrm{T}(A),$$ then there is a c.p.c.~order zero map $\varphi': \mathrm{M}_k(\Comp) \to  A$ such that $$\norm{\varphi' - \varphi} < \sqrt{\eps}$$ and $$ \mathrm{d}_\tau(1_A -\varphi'(1_k)) < \sqrt{\eps},\quad \tau\in\mathrm{T}(A).$$
\end{lem}
\begin{proof}
Since $\varphi$ has order zero, it follows from Theorem 1.2 of \cite{WZ-ndim} that there is $$h\in\mathcal{M}(\textrm{C*}(\varphi(\mathrm{M}_k))) \cap (\textrm{C*}(\varphi(\mathrm{M}_k)))' $$ and a unital homomorphism $$\tilde{\varphi}: \mathrm{M}_k(\Comp) \to \mathcal{M}(\textrm{C*}(\varphi(\mathrm{M}_k))) \cap(h)'$$ such that  $$\varphi(a) = \tilde{\varphi}(a) h.$$ Note that $h=\varphi(1_k)$. 

Let $\tau\in\mathrm{T}(A)$ be arbitrary, and denote by $\mu_\tau$ the probability measure induced by $\tau$ on $\mathrm{sp}(h)\subseteq [0, 1]$. Since $\tau(1_A - h) < \eps$, one has
\begin{eqnarray*}
1-\eps & < & \int_{(0, 1]} t \mathrm{d} \mu_\tau = \int_{(0, 1-\sqrt{\eps}]} t \mathrm{d} \mu_\tau + \int_{(1-\sqrt{\eps}, 1]} t \mathrm{d} \mu_\tau \\ 
& \leq & (1-\sqrt{\eps})\mu_\tau([0, 1-\sqrt{\eps}]) + (1- \mu_\tau([0, 1-\sqrt{\eps}])),
\end{eqnarray*} 
and hence
$$ \mu_\tau([0, 1-\sqrt{\eps}]) < \sqrt{\eps}.$$

Set $f(t) = \min\{\frac{t}{1-\sqrt{\eps}}, 1\}$. Consider $f(h)$ and the c.p.c.~order zero map $$\varphi':=\tilde{\varphi}(a)f(h),\quad a\in\mathrm{M}_k(\Comp).$$ Note that $\norm{h - f(h)} < \sqrt{\eps}$; one has that $$\norm{\varphi - \varphi '} < \eps.$$ On the other hand, for any $\tau\in\mathrm{T}(A)$, one has $$\mathrm{d}_\tau(1 - \varphi'(1_k)) = \mathrm{d}_\tau(1 - f(h)) = \mu_\tau([0, 1-\sqrt{\eps}]) <\sqrt{\eps},$$
as desired.
\end{proof}

\begin{prop}\label{tracial-Z}
Let $(X, \Gamma)$ be a free and minimal dynamical system with the (URP). If $\mathrm{C}(X)\rtimes\Gamma$ is tracially $m$-almost divisible for some $m\in\mathbb N$, then, for any finite set $\{f_1, f_2, ..., f_n\} \subseteq \mathrm{C}(X)\rtimes\Gamma$, any $\eps>0$, and any $k\in\mathbb{N}$, there is a c.p.c.~order zero map $\phi: \mathrm{M}_k(\Comp) \to \mathrm{C}(X)\rtimes\Gamma$ such that
\begin{enumerate}
\item $\norm{[\phi(a), f_i]} < \eps$, $a\in\mathrm{M}(\Comp)$ with $\norm{a} = 1$ and $1\leq i\leq n$, and
\item $\mathrm{d}_\tau(1_A-\phi(1_k)) < \eps$, $\tau\in\mathrm{T}(A)$.
\end{enumerate}
\end{prop}
\begin{proof}
Denote by $A=\mathrm{C}(X)\rtimes\Gamma$. 
By Lemma \ref{pert}, it is enough to show that for any given $\eps>0$ and any finite set $\{f_1, f_2, ..., f_n\}\subseteq A$, there is a c.p.c.~order-zero map $\phi: \mathrm{M}_k(\Comp) \to A$ such that
\begin{enumerate}
\item $\norm{[\phi(a), f_i]} < \eps$, $a\in\mathrm{M}_k(\Comp)$ with $\norm{a} = 1$ and $1\leq i\leq n$, and
\item $\tau(1_A-\phi(1_k)) < \eps$, $\tau\in\mathrm{T}(A)$.
\end{enumerate}

Since order zero maps from $\mathrm{M}_k(\Comp)$ are weakly stable (see Proposition 2.5 of \cite{KW-DR-QD}), one is able to pick $\delta>0$ sufficiently small such that if a c.p.c.~map $\rho: \mathrm{M}_k(\Comp) \to A$ satisfies $$a \perp b \Rightarrow \norm{\rho(a)\rho(b)} <\delta,\quad a, b\in \mathrm{M}_k(\Comp),\ \norm{a} = \norm{b} = 1,$$ there is a c.p.c~order zero map $\theta: \mathrm{M}_k(\Comp) \to A$ such that $$\norm{\rho(a) - \theta(a)} < \frac{\eps}{4},\quad a\in\mathrm{M}_k(\Comp),\ \norm{a} = 1.$$

By Lemma \ref{cut-down}, there are  positive elements $f_1^{0}, f_1^{(1)}, f_2^{0}, f_2^{(1)}, ..., f_n^{0}, f_n^{(1)} \in A$, a C*-algebra $B \subseteq A $ with $B \cong \bigoplus_{s=1}^S \mathrm{M}_{k_s}(\mathrm{C}_0(Z_s))$ for some locally compact Hausdorff spaces $Z_s$, $s=1, ..., S$, a positive element $h\in A$ with norm $1$ such that
\begin{equation}\label{decp} \norm{f_i - (f_i^{0} + f_i^{(1)})} < \frac{\eps}{4},\quad 1\leq i\leq n,\end{equation}
\begin{equation}\label{in-B} h\in B \quad\mathrm{and} \quad f_i^{(1)} \in B,\quad 1\leq i\leq n, \end{equation}
\begin{equation}\label{am-perp} \norm{f_i^{(0)} h^{\frac{1}{2}}} <\frac{\eps}{16}, \quad 1\leq i\leq n, \end{equation}
\begin{equation}\label{comm-0} \norm{[f_i^{(1)}, h^{\frac{1}{2}}]} < \frac{\eps}{24},\quad 1\leq i\leq n,\end{equation} 
and
\begin{equation}\label{small-tr-0} \tau(1-h) < \frac{\eps}{4},\quad \tau\in\mathrm{T}(A).\end{equation}

Consider the unitization $\tilde{B}=B+\Comp 1_A$, and note that $$\tilde{B}\cong \{ f\in \mathrm{C}(\{\infty\} \cup \bigsqcup_{s=1}^S Z_s, \bigoplus_{s=1}^S\mathrm{M}_{k_s}(\Comp)): f(\infty)\in \Comp 1\}.$$ Since the space $\{\infty\} \cup \bigsqcup_{s=1}^S Z_s$ is an inverse limit of finite dimensional CW-complexes, with a small perturbation of $f_1^{(1)}, f_2^{(1)}$, ..., $f_n^{(1)}$, and $h$, one may assume that $\tilde{B}$ (and $B$) has finite nuclear dimension. 

Since $A$ is assumed to be tracially $m$-divisible, applying Lemma 5.11 of \cite{Winter-Z-stable-02} to $\tilde{B}$ and using \eqref{in-B}, one obtains a c.p.c.~order zero map $\varphi: \mathrm{M}_k(\Comp) \to A$ such that
\begin{equation}\label{comm-1}
\norm{[\varphi(a), f_i^{(1)}]} < \frac{\eps}{24},\quad 1\leq i\leq n,\ a\in\mathrm{M}_k(\Comp), \ \norm{a}=1,
\end{equation}
\begin{equation}\label{comm-2}
\norm{[\varphi(a), h]} < \delta,\quad a\in\mathrm{M}_k(\Comp), \ \norm{a}=1,
\end{equation}
and
\begin{equation}\label{small-tr-1}
\tau(1_A - \varphi(1_k)) < \frac{\eps}{4},\quad \tau\in\mathrm{T}(A).
\end{equation}

Consider the c.p.c.~map $$\mathrm{M}_k(\Comp) \ni a \mapsto h^{\frac{1}{2}}\varphi(a)h^{\frac{1}{2}} \in A.$$ Then, for any elements $a, b\in \mathrm{M}_k(\Comp)$ with $a \perp b$ and $\norm{a} = \norm{b} =1$, one has (by \eqref{comm-2})
$$(h^{\frac{1}{2}}\varphi(a)h^{\frac{1}{2}})(h^{\frac{1}{2}}\varphi(b)h^{\frac{1}{2}}) = h^{\frac{1}{2}}\varphi(a)h \varphi(b)h^{\frac{1}{2}} \approx_{\delta} h^{\frac{3}{2}}\varphi(a) \varphi(b)h^{\frac{1}{2}} =0, $$ and hence, by the choice of $\delta$, there exists a c.p.c~order zero map $\phi:\mathrm{M}_k(\Comp) \to A$ such that 
\begin{equation}\label{pert-od0}
\norm{\phi(a) - h^{\frac{1}{2}}\varphi(a)h^{\frac{1}{2}}} < \frac{\eps}{4},\quad a\in \mathrm{M}_k(\Comp), \ \norm{a}=1.
\end{equation} 
Then, for any $a\in \mathrm{M}_k(\Comp)$ with $\norm{a} = 1$ and any $1\leq i\leq n$, one has
\begin{eqnarray*}
\norm{[\phi(a), f_i]} & < & \norm{[h^{\frac{1}{2}}\varphi(a)h^{\frac{1}{2}}, f_i]} + \frac{\eps}{2} \quad\quad \textrm{(by \eqref{pert-od0})} \\
&<& \norm{[h^{\frac{1}{2}}\varphi(a)h^{\frac{1}{2}}, f_i^{(0)} + f_i^{(1)}]} + \frac{3\eps}{4} \quad\quad\textrm{(by \eqref{decp})} \\
& = & \norm{[h^{\frac{1}{2}}\varphi(a)h^{\frac{1}{2}}, f_i^{(0)}]} + \norm{[h^{\frac{1}{2}}\varphi(a)h^{\frac{1}{2}}, f_i^{(1)}]} + \frac{3\eps}{4} \\
& < & \frac{\eps}{8} + \frac{\eps}{8} + \frac{3\eps}{4} = \eps \quad\quad\textrm{(by \eqref{am-perp}, \eqref{comm-0} and \eqref{comm-1}).}
\end{eqnarray*}
Moreover, applying Lemma \ref{L-trace} with \eqref{small-tr-0} and \eqref{small-tr-1}, together with \eqref{pert-od0}, one has $$\tau(\phi(1_k)) \approx_{\frac{\eps}{4}} \tau(h^{\frac{1}{2}}\varphi(1_k)h^{\frac{1}{2}}) = \tau(h\varphi(1_k)) > 1-\frac{\eps}{2},\quad \tau\in\mathrm{T}(A),$$
as desired.
\end{proof}

\begin{thm}\label{Z-stable}
Let $(X, \Gamma)$ be a free and minimal dynamical system with the (URP) and (COS). If $(X, \Gamma)$ has mean dimension zero, then $(\mathrm{C}(X)\rtimes\Gamma) \otimes \mathcal Z \cong \mathrm{C}(X) \rtimes \Gamma$. 

In particular, let $(X_1, \Gamma_1)$ and $(X_2, \Gamma_2)$ be two free minimal dynamical systems with the (URP), Cuntz comparison of open sets, and zero mean dimension, then  $$\mathrm{C}(X_1)\rtimes\Gamma_1 \cong \mathrm{C}(X_2)\rtimes\Gamma_2$$ if and only if $$\mathrm{Ell}(\mathrm{C}(X_1)\rtimes\Gamma_1) \cong \mathrm{Ell}(\mathrm{C}(X_2)\rtimes\Gamma_2),$$ where $\mathrm{Ell}(\cdot) = (\Kzero(\cdot), \Kzero^+(\cdot), [1], \mathrm{T}(\cdot), \rho, \Kone(\cdot))$ is the Elliott invariant. Moreover, these C*-algebras are inductive limits of unital subhomogeneous C*-algebras.
\end{thm}

\begin{proof}
It follows from Corollary \ref{md0-divisible} that $\mathrm{C}(X)\rtimes\Gamma$ is tracially $0$-divisible. It follows from Theorem 4.8 of \cite{Niu-MD-Z} that $\mathrm{C}(X)\rtimes\Gamma$ has strict comparison of positive elements. Together with Proposition \ref{tracial-Z}, one has that $\mathrm{C}(X)\rtimes\Gamma$ is tracially $\mathcal Z$-stable. Since $\mathrm{C}(X)\rtimes\Gamma$ is nuclear,  it is $\mathcal Z$-stable as desired.
\end{proof}

\begin{cor}\label{Z-corollary}
Let $(X, \Gamma)$ be a free and minimal dynamical system with mean dimension zero. Assume that
\begin{itemize}
\item either $\Gamma = \Int^d$ for some $d\geq 1$, or
\item $(X, \Gamma)$ is an extension of a Cantor system and $\Gamma$ has subexponetial growth.
\end{itemize}
Then, the C*-algebra $\mathrm{C}(X) \rtimes \Gamma$ is classified by the Elliott invariant and is an inductive limit of unital subhomogeneous C*-algebras. 
\end{cor}

\begin{proof}
It follows from \cite{Niu-MD-Z} and \cite{Niu-MD-Zd} that the dynamical systems being considered have the (URP) and (COS). The statement then follows from Theorem \ref{Z-stable}.
\end{proof}

\section{An alternative approach: $\mathrm{mdim}0 + \mathrm{URP} \Rightarrow \mathrm{SBP}$}\label{SBP}

In this section, one considers the zero mean dimension together with the (URP), and shows that these two conditions actually implies that the dynamical system has the small boundary property (SBP). Together with \cite{KS-comparison} and \cite{Niu-MD-Z}, this gives another proof of Theorem \ref{Z-stable}.

\begin{thm}\label{md0-2-sbp}
Let $(X, \Gamma)$ be a free topological dynamical system with the (URP). If $$\mathrm{mdim}(X, \Gamma) = 0,$$ then $(X, \Gamma)$ has the small boundary property.
\end{thm}
\begin{proof}
It follows from Lemma 5.5 and Theorem 5.3 of \cite{GLT-Zk} that, in order to show that $(X, \Gamma)$ has the (SBP), it is enough to show that for any continuous function $f: X \to\mathbb R$ and any $\eps>0$, there is a continuous function $g: X \to \mathbb R$ such that 
\begin{enumerate}
\item $\norm{f - g}_\infty < \eps$, and
\item $\mathrm{ocap}(\{x\in X: g(x) = 0\}) < \eps$.
\end{enumerate}

Let $f: X \to\mathbb R$ and $\eps>0$ be given. Pick $\mathcal U$ to be a finite open cover of $X$ such that
$$\abs{f(x) - f(y)} < \frac{\eps}{3} ,\quad x, y\in U, U\in\mathcal U.$$

Since $\mathrm{mdim}(X, \Gamma) = 0$, there is $(K, \eps')$, where $K\subseteq \Gamma$ is a finite set and $\eps'>0$, such that if $\Gamma_0\subseteq \Gamma$ is $(K, \eps')$-invariant, there is an open cover $\mathcal V$ such that
\begin{enumerate}
\item\label{a} $\mathcal V$ refines $ \bigvee_{\gamma\in\Gamma_0} \mathcal U\gamma$, and
\item\label{b} $\mathrm{ord}(\mathcal V) < \frac{\eps}{3}\abs{\Gamma_0}$.
\end{enumerate}

Since $(X, \Gamma)$ has the (URP), there are closed sets $B_1, B_2, ..., B_S$ and $(K, \eps')$-invariant sets $\Gamma_1, \Gamma_2, ..., \Gamma_S \subseteq\Gamma$ such that 
$$B_s\gamma,\quad \gamma\in\Gamma_s,\  1\leq s\leq S$$
are mutually disjoint and
$$\mathrm{ocap}(X \setminus \bigsqcup_{s=1}^{S}\bigsqcup_{\gamma\in\Gamma_s} B_s\gamma) < \frac{\eps}{3}.$$

Pick a small neighborhood $U_s$ of each $B_s$, $s=1, 2, ..., S$, such that $$U_s\gamma,\quad \gamma\in\Gamma_s,\  1\leq s\leq S,$$
are still mutually disjoint. Note that 
$$\mathrm{ocap}(X \setminus \bigsqcup_{s=1}^{S}\bigsqcup_{\gamma\in\Gamma_s} U_s\gamma) \leq \mathrm{ocap}(X \setminus \bigsqcup_{s=1}^{S}\bigsqcup_{\gamma\in\Gamma_s} B_s\gamma)   < \frac{\eps}{3}.$$

For each $s=1, 2, ..., S$, since $\Gamma_s$ is $(K, \eps')$-invariant, there is an open cover $\mathcal V$ of $X$ such that 
\begin{enumerate}
\item\label{a} $\mathcal V$ refines $ \bigvee_{\gamma\in\Gamma_0} \mathcal U\gamma$, and
\item\label{b} $\mathrm{ord}(\mathcal V) < \frac{\eps}{3} \abs{\Gamma_s}$.
\end{enumerate}
Then, consider the collection of open sets
$$\mathcal V_s:=\{V \cap U_s: V\in\mathcal V\}.$$
Note that $\mathcal V_s$ covers $B_s$ and for any $V\in\mathcal V_s$ and any $\gamma\in \Gamma_s$, there is $U\in\mathcal U$ such that $$V\gamma \subseteq U.$$

For $\mathcal V_s$, pick continuous functions $$\phi_V^{(s)}: X \to [0, 1],\quad V\in \mathcal V_s$$ such that
$$(\phi_V^{(s)})^{-1}((0, 1]) \subseteq V,$$ 
$$\sum_{V \in\mathcal V_s} \phi_V^{(s)}(x) \leq 1,\quad x \in X,\quad \mathrm{and}$$
$$\sum_{V \in\mathcal V_s} \phi_V^{(s)}(x) = 1,\quad x\in B_s.$$

For $\mathcal V_s$, also consider the simplicial complex $\Delta_s$ spanned by $[V]$, $V\in \mathcal V_s$, with $$[V_0], [V_1], ...,[V_d]$$ span a simplex if and only if $$V_0\cap V_1 \cap \cdots \cap V_d \neq \varnothing.$$ Note that
\begin{equation}\label{small-ratio}
\mathrm{dim}(\Delta_s) = \mathrm{ord}(\mathcal V_s) \leq \mathrm{ord}(\mathcal V) \leq \frac{\eps}{3}\abs{\Gamma_s}.
\end{equation}
Define the map
$$\eta_s: X \ni x \mapsto \sum_{V \in \mathcal V_s}\phi_{V}^{(s)}(x) [V] \in \mathrm{C}\Delta_s,$$
where $\mathrm{C}\Delta_s$ is the cone over $\Delta_s$. Note that $$\eta_s(B_s) \subseteq \Delta_s.$$

For each $V\in\mathcal V_s$, pick a point $x^*_V\in V$, and define
$$\tilde{f} = f(1-\sum_{s=1}^S\sum_{\gamma\in\Gamma_s}\sum_{V\in\mathcal V_s}(\phi_V^{(s)}\circ\gamma^{-1})) + \sum_{s=1}^S\sum_{\gamma\in\Gamma_s}\sum_{V\in\mathcal V_s}f(x_V^*\gamma)(\phi_V^{(s)}\circ\gamma^{-1}).$$ Then, for any $x\in X$,
\begin{eqnarray*}
 & & \abs{f(x) - \tilde{f}(x)} \\
 & = &\abs{f(x) -  (f(x)(1-\sum_{s=1}^S\sum_{\gamma\in\Gamma_s}\sum_{V\in\mathcal V_s}\phi_V^{(s)}(x\gamma^{-1})) + \sum_{s=1}^S\sum_{\gamma\in\Gamma_s}\sum_{V\in\mathcal V_s}f(x_V^*\gamma)\phi_V^{(s)}(x\gamma^{-1}))} \\
 & = & \left| f(x) (1-\sum_{s=1}^S\sum_{\gamma\in\Gamma_s}\sum_{V\in\mathcal V_s}\phi_V^{(s)}(x\gamma^{-1}) + \sum_{s=1}^S\sum_{\gamma\in\Gamma_s}\sum_{V\in\mathcal V_s}\phi_V^{(s)}(x\gamma^{-1})) -  \right. \\
 & & \left. (f(x)(1-\sum_{s=1}^S\sum_{\gamma\in\Gamma_s}\sum_{V\in\mathcal V_s}\phi_V^{(s)}(x\gamma^{-1})) + \sum_{s=1}^S\sum_{\gamma\in\Gamma_s}\sum_{V\in\mathcal V_s}f(x_V^*\gamma)\phi_V^{(s)}(x\gamma^{-1})) \right| \\
 & = & \abs{\sum_{s=1}^S\sum_{\gamma\in\Gamma_s}\sum_{V\in\mathcal V_s} (f(x) - f(x_V^*\gamma)) \phi_V^{(s)}(x\gamma^{-1}))} \\
 & \leq & \sum_{s=1}^S\sum_{\gamma\in\Gamma_s}\sum_{V\in\mathcal V_s}\abs{f(x) - f(x_V^*\gamma)} \phi_V^{(s)}(x\gamma^{-1})) \\
 & < & \frac{\eps}{3}.
\end{eqnarray*}
That is, 
\begin{equation}\label{small-pert-0}
\norm{f - \tilde{f}} < \frac{\eps}{3}.
\end{equation}

Define piecewise linear function $F_s: \mathrm{C}\Delta_s \to \Real^{\abs{\Gamma_s}}$ by
$$F_s([V]) = \bigoplus_{\gamma \in \Gamma_s}f(x_V^*\gamma) \in \Real^{\abs{\Gamma_s}}.$$ Then
$$\tilde{f} = f(1-\sum_{s=1}^S\sum_{\gamma\in\Gamma_s}\sum_{V\in\mathcal V_s}(\phi_V^{(s)}\circ\gamma^{-1})) + \sum_{s=1}^S\sum_{\gamma\in \Gamma_s}\pi_{s, \gamma} \circ F_s\circ\eta_s\circ\gamma^{-1},$$
where $\pi_{s, \gamma}$ is the projection of $\Real^{\abs{\Gamma_s}}$ to the $\gamma$-coordinate.

By Lemma 5.7 of \cite{GLT-Zk}, there is a linear map $\tilde{F}_s: \mathrm{C}\Delta_s \to \Real^{\abs{\Gamma_s}}$ such that
$$\norm{F_s - \tilde{F}_s}_\infty < \frac{\eps}{3}$$ and
\begin{equation}\label{small-zero}
\abs{\{\gamma\in\Gamma_s: \pi_{s, \gamma}(\tilde{F}_s(x)) = 0\}} \leq \mathrm{dim}\Delta_s,\quad x\in \mathrm{C}\Delta_s.
\end{equation}

Put $$g = f(1-\sum_{s=1}^S\sum_{\gamma\in\Gamma_s}\sum_{V\in\mathcal V_s}(\phi_V^{(s)}\circ\gamma^{-1})) + \sum_{s=1}^S\sum_{\gamma\in \Gamma_s}\pi_{s, \gamma} \circ \tilde{F}_s\circ\eta_s\circ\gamma^{-1},$$ and then, for any $x\in X$, 
\begin{eqnarray*}
& & \abs{\tilde{f}(x) - g(x)} \\
& = & \abs{\sum_{s=1}^S\sum_{\gamma\in \Gamma_s}\pi_{s, \gamma} \circ F_s\circ\eta_s(x\gamma^{-1}) - \sum_{s=1}^S\sum_{\gamma\in \Gamma_s}\pi_{s, \gamma} \circ \tilde{F}_s\circ\eta_s(x\gamma^{-1})} \\
& = & \abs{\sum_{s=1}^S\sum_{\gamma\in \Gamma_s} (\pi_{s, \gamma} \circ F_s\circ\eta_s(x\gamma^{-1}) - \pi_{s, \gamma} \circ \tilde{F}_s\circ\eta_s(x\gamma^{-1})  )  }.
\end{eqnarray*}
If $x\notin\bigsqcup_{s=1}^S \bigsqcup_{\gamma\in \Gamma_s}U_s\gamma$, then $$\eta_s(x\gamma^{-1}) = \mathbf{0},\quad \gamma\in\Gamma_s,\ s=1, ..., S.$$ Hence $$\pi_{s, \gamma} \circ F_s\circ\eta_s(x\gamma^{-1}) = \pi_{s, \gamma} \circ \tilde{F}_s\circ\eta_s(x\gamma^{-1}) =0,\quad \gamma\in\Gamma_s,\ s=1, ..., S,$$ and 
\begin{equation}\label{small-pert-1}
\tilde{f}(x) = g(x).
\end{equation}
If $x\in\bigsqcup_{s=1}^S \bigsqcup_{\gamma\in \Gamma_s}U_s\gamma$, then there exist $s_0\in\{1, ..., S\}$ and $\gamma_0\in\Gamma_{s_0}$ such that 
$$\textrm{$x$ is only in $U_{s_0}\gamma_0$}.$$ 
Then
$$
\eta_s(x\gamma^{-1}) = \mathbf{0},\quad \gamma\in\Gamma_s,\ s \neq s_0,
$$
and
$$ 
\eta_{s_0}(x\gamma^{-1}) = \mathbf{0},\quad \gamma\neq\gamma_0.
$$ 
Hence
\begin{eqnarray*}
&&\abs{\sum_{s=1}^S\sum_{\gamma\in \Gamma_s} (\pi_{s, \gamma} \circ F_s\circ\eta_s(x\gamma^{-1}) - \pi_{s, \gamma} \circ \tilde{F}_s\circ\eta_s(x\gamma^{-1})  )  } \\
&=& \abs{\pi_{s_0, \gamma_0} \circ F_{s_0}\circ\eta_{s_0}(x\gamma_0^{-1}) - \pi_{s_0, \gamma_0} \circ \tilde{F}_{s_0}\circ\eta_{s_0}(x\gamma_0^{-1})   } \\
&<& \frac{\eps}{3},
\end{eqnarray*}
and $$\abs{\tilde{f}(x) - g(x)} < \frac{\eps}{3}.$$ Together with \eqref{small-pert-1}, one has
$$\norm{\tilde{f} - g} < \frac{\eps}{3};$$ together with \eqref{small-pert-0}, one has
$$\norm{f - g} < \frac{2\eps}{3}<\eps.$$

Let us estimate $\mathrm{ocap}(\{x\in X: g(x) = 0\})$. Fist, note that for an arbitrary $x\in B_s$, where $s\in\{1, 2, ..., S\}$,  by \eqref{small-zero}, 
\begin{equation}\label{small-dim}
\{\gamma\in\Gamma_s: g(x\gamma) = 0\} =  \{\gamma\in\Gamma_s: \sum_{s=1}^S\sum_{\gamma\in \Gamma_s}\pi_{s, \gamma}(\tilde{F}_s(\eta_s(x))) = 0\} 
 \leq \mathrm{dim}\Delta_s. 
\end{equation}

Let $\Gamma_0\subseteq\Gamma$ be a finite set which is sufficiently invariant such that 
\begin{equation}\label{small-boundary}
\frac{\abs{\mathrm{int}_{\bigcup_{s=1}^S(\Gamma^2_{s_i})^{-1}}\Gamma_0}}{\abs{\Gamma_0}}> 1-\frac{\eps}{3},
\end{equation} 
and
\begin{equation}\label{small-lf}
\frac{1}{\abs{\Gamma_0}}\abs{\{\gamma\in\Gamma_0: x\gamma \in X\setminus\bigsqcup_{s=1}^S\bigsqcup_{\gamma\in\Gamma_s} B_s\gamma\}}<\frac{\eps}{3},\quad x\in X.
\end{equation}

Denote by $$K= \bigcup_{s=1}^S\Gamma_s,$$ and note that  
${\mathrm{int}_{\bigcup_{s=1}^S\Gamma^{-1}_s}\Gamma_0} = \Gamma_0\cap(\Gamma_0K)$

Let $x\in X$ be arbitrary, and consider the orbit $x\Gamma$. The partition $$X=(X\setminus\bigsqcup_{c=1}^S\bigsqcup_{\gamma\in\Gamma_s}B_s\gamma) \sqcup \bigsqcup_{c=1}^S\bigsqcup_{\gamma\in\Gamma_s}B_s\gamma$$ induces a partition of $x\Gamma$; since the action is free, this induces a partition of $\Gamma$: 
$$\Gamma = \Lambda \sqcup \bigsqcup_{i=1}^\infty c_i \Gamma_{s(i)},$$
where $s(i)\in\{1, 2, ..., S\}$ for each $i=1, 2, ...$, 
$$\Lambda=\{\gamma\in\Gamma: x\gamma\in X\setminus\bigsqcup_{s=1}^S\bigsqcup_{\gamma\in\Gamma_s} B_s\gamma \},$$
and $c_i\in \Gamma$, $i=1, 2, ...$, satisfying $xc_i \in B_{s(i)}$. Restrict this partition to $\Gamma_0$, one has 
$$\Gamma_0 = (\Gamma_0 \cap \Lambda) \cup \bigsqcup_{c_i\Gamma_{s(i)}\nsubseteq\Gamma_0} (\Gamma_0 \cap (c_i\Gamma_{s(i)})) \cup \bigsqcup_{c_i\Gamma_{s(i)}\subseteq \Gamma_0} c_i\Gamma_{s(i)}.$$

A straightforward calculation shows that if $\gamma\in \Gamma_0\cap (c_i\Gamma_{s(i)})$ and $c_i\Gamma_{s(i)} \nsubseteq \Gamma_0$, then $\gamma \notin \mathrm{int}_{(\Gamma_{s(i)}^2)^{-1}}\Gamma_0$. Therefore
$$ \bigsqcup_{c_i\Gamma_{s(i)}\nsubseteq\Gamma_0} (\Gamma_0 \cap c_i\Gamma_{s(i)}) \subseteq \Gamma_0\setminus\mathrm{int}_{\bigcup_{s=1}^S(\Gamma^2_{s_i})^{-1}}(\Gamma_0) =: \partial_{\bigcup_{s=1}^S(\Gamma^2_{s_i})^{-1}}\Gamma_0,$$
and, by \eqref{small-dim}, \eqref{small-boundary}, \eqref{small-lf}, and \eqref{small-ratio},
\begin{eqnarray*}
& & \frac{1}{\Gamma_0}\abs{\{\gamma\in\Gamma_0: g(x\gamma) = 0\}} \\
& \leq & \frac{\abs{\Gamma_0\cap\Lambda}}{\abs{\Gamma_0}} +  \frac{\abs{\partial_{\bigcup_{s=1}^S(\Gamma^2_{s_i})^{-1}}\Gamma_0}}{\abs{\Gamma_0}} + \frac{1}{\abs{\Gamma_0}}\sum_{c_i\Gamma_{s(i)}\subseteq\Gamma_0} \mathrm{dim}\Delta_{s(i)} \\
& \leq & \frac{\eps}{3} + \frac{\eps}{3} + \frac{\sum_{c_i\Gamma_{s(i)}\subseteq\Gamma_0} \mathrm{dim}\Delta_{s(i)}}{\sum_{c_i\Gamma_{s(i)}\subseteq\Gamma_0} \abs{\Gamma_{s(i)}}} \\
& \leq & \frac{2\eps}{3} + \frac{\eps}{3} = \eps,
\end{eqnarray*}
as desired.
\end{proof}

\begin{rem}
Note that if $\Gamma = \Int^d$, it follows from Theorem 1.10.1 and Theorem 1.10.3 of \cite{GY-ETDS-2011} that $$\mathrm{TRP} + \mathrm{mdim}0 \Leftrightarrow \mathrm{SBP},$$ where TRP stands for the Topological Rokhlin Property in the sense of 1.9 of \cite{GY-ETDS-2011} ($\mathrm{edim}(X, \Int^d)\leq l$ densely for some $l\in\mathbb N$ is actually not needed in Theorem 1.10.3). It is easy to see that URP implies TRP. Therefore, in this case, the statement of Theorem \ref{md0-2-sbp} is covered by Theorem 1.10.3 of \cite{GY-ETDS-2011}. It was also proved later in \cite{GLT-Zk} (Corollary 5.4) that $$ \mathrm{mdim}0 \Leftrightarrow \mathrm{SBP}$$ for any $\Int^d$-actions with marker property.
\end{rem}

With the Uniform Property Gamma and \cite{CETWW-CP}, Kerr and Szabo has the following: 
\begin{thm}[Corollary 9.5 of \cite{KS-comparison}]\label{KS}
Assume that $(X, \Gamma)$ has the small boundary property. Then, $\mathrm{C}(X)\rtimes\Gamma$ has the strict comparison if and only if it is $\mathcal Z$-stable.
\end{thm}

Thus, together with Theorem \ref{md0-2-sbp} and Theorem 4.8 of \cite{Niu-MD-Z}, one has the following:
\begin{proof}[Alternative proof of Theorem \ref{Z-stable}]
Since $(X,  \Gamma)$ is assumed to have the (URP), by Theorem \ref{md0-2-sbp}, it has the (SBP) since it has mean dimension zero. Therefore, by Theorem \ref{KS}, in order to prove the theorem,  it is enough to show that $C(X)\rtimes\Gamma$ has the strict comparison of positive elements. But since $(X, \Gamma)$ has the (COS)  and zero mean dimension, the strict comparison of $\mathrm{C}(X)\rtimes\Gamma$ follows from Theorem 4.8 of \cite{Niu-MD-Z}.
\end{proof}


\bibliographystyle{plainurl}
\bibliography{operator_algebras}

\begin{thebibliography}{10}

\bibitem{BPT-Cuntz}
N.~P. Brown, F.~Perera, and A.~S. Toms.
\newblock The {C}untz semigroup, the {E}lliott conjecture, and dimension
  functions on {$C^*$}-algebras.
\newblock {\em J. Reine Angew. Math.}, 621:191--211, 2008.
\newblock URL: \url{http://dx.doi.org/10.1515/CRELLE.2008.062}, \href
  {https://doi.org/10.1515/CRELLE.2008.062}
  {\path{doi:10.1515/CRELLE.2008.062}}.

\bibitem{CETWW-CP}
J.~Castillejos, S.~Evington, A.~Tikuisis, S.~White, and W.~Winter.
\newblock Strict closures of nuclear {C*}-algebras and uniform property
  {$\Gamma$}.
\newblock {\em In preparation.}

\bibitem{Effros-Hahn}
E.~G. Effros and F.~Hahn.
\newblock {\em Locally compact transformation groups and {C*}- algebras}.
\newblock Memoirs of the American Mathematical Society, No.~75. American
  Mathematical Society, Providence, R.I., 1967.

\bibitem{EGLN-ASH}
G.~A. Elliott, G.~Gong, H.~Lin, and Z.~Niu.
\newblock The classification of simple separable unital $\textrm{$\mathcal
  Z$}$-stable locally {ASH} algebras.
\newblock {\em J. Funct. Anal.}, (12):5307--5359, 2017.
\newblock URL: \url{http://dx.doi.org/10.1016/j.jfa.2017.03.001}.

\bibitem{EN-MD0}
G.~A. Elliott and Z.~Niu.
\newblock The {C{$^*$}}-algebra of a minimal homeomorphism of zero mean
  dimension.
\newblock {\em Duke Math. J.}, 166(18):3569--3594, 2017.
\newblock \href {https://doi.org/10.1215/00127094-2017-0033}
  {\path{doi:10.1215/00127094-2017-0033}}.

\bibitem{ENST-ASH}
G.~A. Elliott, Z.~Niu, L.~Santiago, and A.~Tikuisis.
\newblock Decomposition rank of approximately subhomogeneous {C*}-algebras.
\newblock 05 2015.
\newblock URL: \url{http://arxiv.org/abs/1505.06100}, \href
  {http://arxiv.org/abs/1505.06100} {\path{arXiv:1505.06100}}.

\bibitem{GK-Dyn}
J.~Giol and D.~Kerr.
\newblock Subshifts and perforation.
\newblock {\em J. Reine Angew. Math.}, 639:107--119, 2010.
\newblock URL: \url{http://dx.doi.org/10.1515/CRELLE.2010.012}, \href
  {https://doi.org/10.1515/CRELLE.2010.012}
  {\path{doi:10.1515/CRELLE.2010.012}}.

\bibitem{Gromov-MD}
M.~Gromov.
\newblock Topological invariants of dynamical systems and spaces of holomorphic
  maps. {I}.
\newblock {\em Math. Phys. Anal. Geom.}, 2(4):323--415, 1999.
\newblock URL: \url{https://mathscinet.ams.org/mathscinet-getitem?mr=1742309},
  \href {https://doi.org/10.1023/A:1009841100168}
  {\path{doi:10.1023/A:1009841100168}}.

\bibitem{GY-ETDS-2011}
Y.~Gutman.
\newblock Embedding {$\Bbb Z^k$}-actions in cubical shifts and {$\Bbb
  Z^k$}-symbolic extensions.
\newblock {\em Ergodic Theory Dynam. Systems}, 31(2):383--403, 2011.
\newblock URL: \url{https://mathscinet.ams.org/mathscinet-getitem?mr=2776381},
  \href {https://doi.org/10.1017/S0143385709001096}
  {\path{doi:10.1017/S0143385709001096}}.

\bibitem{GLT-Zk}
Y.~Gutman, E.~Lindenstrauss, and M.~Tsukamoto.
\newblock Mean dimension of {$\Bbb{Z}^k$}-actions.
\newblock {\em Geom. Funct. Anal.}, 26(3):778--817, 2016.
\newblock URL: \url{http://dx.doi.org/10.1007/s00039-016-0372-9}, \href
  {https://doi.org/10.1007/s00039-016-0372-9}
  {\path{doi:10.1007/s00039-016-0372-9}}.

\bibitem{HO-Z}
I.~Hirshberg and J.~Orovitz.
\newblock Tracially {$\mathcal Z$}-absorbing {C*}-algebras.
\newblock {\em J. Funct. Anal.}, 265(5):765--785, 2013.
\newblock URL: \url{https://mathscinet.ams.org/mathscinet-getitem?mr=3063095},
  \href {https://doi.org/10.1016/j.jfa.2013.05.005}
  {\path{doi:10.1016/j.jfa.2013.05.005}}.

\bibitem{KS-comparison}
D.~Kerr and G.~Szabo.
\newblock Almost finiteness and the small boundary property.
\newblock 07 2018.
\newblock URL: \url{https://arxiv.org/pdf/1807.04326}, \href
  {http://arxiv.org/abs/1807.04326} {\path{arXiv:1807.04326}}.

\bibitem{KR-CenSeq}
E.~Kirchberg and M.~R{\o}rdam.
\newblock Central sequence {C*}-algebras and tensorial absorption of the
  {Jiang-Su} algebra.
\newblock {\em Crelles Journal}, to appear.
\newblock URL: \url{http://arxiv.org/abs/1209.5311}, \href
  {http://arxiv.org/abs/1209.5311} {\path{arXiv:1209.5311}}.

\bibitem{KW-DR-QD}
E.~Kirchberg and W.~Winter.
\newblock Covering dimension and quasidiagonality.
\newblock {\em Internat. J. Math.}, 15(1):63--85, 2004.
\newblock URL: \url{http://dx.doi.org/10.1142/S0129167X04002119}, \href
  {https://doi.org/10.1142/S0129167X04002119}
  {\path{doi:10.1142/S0129167X04002119}}.

\bibitem{LP-Dym}
H.~Lin and N.~C. Phillips.
\newblock Crossed products by minimal homeomorphisms.
\newblock {\em J. Reine Angew. Math.}, 641:95--122, 2010.
\newblock URL: \url{http://dx.doi.org/10.1515/CRELLE.2010.029}, \href
  {https://doi.org/10.1515/CRELLE.2010.029}
  {\path{doi:10.1515/CRELLE.2010.029}}.

\bibitem{QL-Ph-min-diff}
Q.~Lin and N.~C. Phillips.
\newblock Direct limit decomposition for {C*}-algebras of minimal
  diffeomorphisms.
\newblock In {\em Operator algebras and applications}, volume~38 of {\em Adv.
  Stud. Pure Math.}, pages 107--133. Math. Soc. Japan, Tokyo, 2004.

\bibitem{Lind-MD}
E.~Lindenstrauss.
\newblock Mean dimension, small entropy factors and an embedding theorem.
\newblock {\em Inst. Hautes {\'E}tudes Sci. Publ. Math.}, (89):227--262 (2000),
  1999.
\newblock URL: \url{http://www.numdam.org/item?id=PMIHES_1999__89__227_0}.

\bibitem{Lindenstrauss-Weiss-MD}
E.~Lindenstrauss and B.~Weiss.
\newblock Mean topological dimension.
\newblock {\em Israel J. Math.}, 115:1--24, 2000.
\newblock URL: \url{http://dx.doi.org/10.1007/BF02810577}, \href
  {https://doi.org/10.1007/BF02810577} {\path{doi:10.1007/BF02810577}}.

\bibitem{Matui-Sato-CP}
H.~Matui and Y.~Sato.
\newblock Strict comparison and {$\mathcal Z$}-absorption of nuclear
  {C*}-algebras.
\newblock {\em Acta Math.}, 209(1):179--196, 2012.
\newblock URL: \url{http://dx.doi.org/10.1007/s11511-012-0084-4}, \href
  {https://doi.org/10.1007/s11511-012-0084-4}
  {\path{doi:10.1007/s11511-012-0084-4}}.

\bibitem{Niu-MD-Zd}
Z.~Niu.
\newblock Comparison radius and mean topological dimension: $\mathbb
  {Z}^d$-actions.
\newblock {\em arXiv:1906.09171}, 2019.

\bibitem{Niu-MD-Z}
Z.~Niu.
\newblock Comparison radius and mean topological dimension: {R}okhlin property,
  comparison of open sets, and subhomogeneous {C*}-algebras.
\newblock {\em arXiv:1906.09172}, 2019.

\bibitem{NW-TAF}
Z.~Niu and Q.Y. Wang.
\newblock A tracially {AF} algebra which is not {$\mathcal Z$}-absorbing.
\newblock {\em arXiv: 1902.03325}, 2019.

\bibitem{Put-PJM}
I.~F. Putnam.
\newblock The {C*}-algebras associated with minimal homeomorphisms of the
  {C}antor set.
\newblock {\em Pacific J. Math.}, 136(2):329--353, 1989.
\newblock URL:
  \url{http://projecteuclid.org/getRecord?id=euclid.pjm/1102650733}.

\bibitem{Sato-CP}
Y.~Sato.
\newblock Trace spaces of simple nuclear {C*}-algebras with finite-dimensional
  extreme boundary.
\newblock {\em Preprint}, 2012.
\newblock URL: \url{http://arxiv.org/abs/1209.3000}, \href
  {http://arxiv.org/abs/1209.3000} {\path{arXiv:1209.3000}}.

\bibitem{Karen-sphere}
K.~R. Strung.
\newblock {C*}-algebras of minimal dynamical systems of the product of a
  {C}antor set and an odd dimensional sphere.
\newblock {\em J. Funct. Anal.}, 268(3):671--689, 2015.
\newblock URL: \url{http://dx.doi.org/10.1016/j.jfa.2014.10.014}, \href
  {https://doi.org/10.1016/j.jfa.2014.10.014}
  {\path{doi:10.1016/j.jfa.2014.10.014}}.

\bibitem{Szabo-Z}
G.~Szab\'{o}.
\newblock The {R}okhlin dimension of topological {$\Bbb{Z}^m$}-actions.
\newblock {\em Proc. Lond. Math. Soc. (3)}, 110(3):673--694, 2015.
\newblock URL: \url{https://mathscinet.ams.org/mathscinet-getitem?mr=3342101},
  \href {https://doi.org/10.1112/plms/pdu065} {\path{doi:10.1112/plms/pdu065}}.

\bibitem{Toms-SDG}
A.~S. Toms.
\newblock {K}-theoretic rigidity and slow dimension growth.
\newblock {\em Invent. Math.}, 183(2):225--244, 2011.
\newblock URL: \url{http://dx.doi.org/10.1007/s00222-010-0273-8}, \href
  {https://doi.org/10.1007/s00222-010-0273-8}
  {\path{doi:10.1007/s00222-010-0273-8}}.

\bibitem{TWW-Z}
A.~S. Toms, S.~White, and W.~Winter.
\newblock {$\mathcal{Z}$}-stability and finite-dimensional tracial boundaries.
\newblock {\em Int. Math. Res. Not. IMRN}, (10):2702--2727, 2015.
\newblock URL: \url{http://dx.doi.org/10.1093/imrn/rnu001}, \href
  {https://doi.org/10.1093/imrn/rnu001} {\path{doi:10.1093/imrn/rnu001}}.

\bibitem{TW-Dym-1}
A.~S. Toms and W.~Winter.
\newblock Minimal dynamics and {K}-theoretic rigidity: {E}lliott's conjecture.
\newblock {\em Geom. Funct. Anal.}, 23(1):467--481, 2013.
\newblock URL: \url{http://dx.doi.org/10.1007/s00039-012-0208-1}, \href
  {https://doi.org/10.1007/s00039-012-0208-1}
  {\path{doi:10.1007/s00039-012-0208-1}}.

\bibitem{Williams}
D.~P. Williams.
\newblock {\em {Crossed Products of C*-Algebras}}.
\newblock Mathematical Surveys and Monographs, Volume 134. American
  Mathematical Society, Providence, RI, 2007.
\newblock URL: \url{http://dx.doi.org/10.1090/surv/134}, \href
  {https://doi.org/10.1090/surv/134} {\path{doi:10.1090/surv/134}}.

\bibitem{Winter-Z-stable-02}
W.~Winter.
\newblock Nuclear dimension and {$\mathcal{Z}$}-stability of pure
  {C*}-algebras.
\newblock {\em Invent. Math.}, 187(2):259--342, 2012.
\newblock URL: \url{http://dx.doi.org/10.1007/s00222-011-0334-7}, \href
  {https://doi.org/10.1007/s00222-011-0334-7}
  {\path{doi:10.1007/s00222-011-0334-7}}.

\bibitem{Winter-TA}
W.~Winter.
\newblock Classifying crossed product {C*}-algebras.
\newblock {\em Amer. J. Math.}, 138(3):793--820, June 2016.
\newblock URL: \url{DOI:10.1353/ajm.2016.0029}.

\bibitem{WZ-ndim}
W.~Winter and J.~Zacharias.
\newblock The nuclear dimension of {C*}-algebras.
\newblock {\em Adv. Math.}, 224(2):461--498, 2010.
\newblock URL: \url{http://dx.doi.org/10.1016/j.aim.2009.12.005}, \href
  {https://doi.org/10.1016/j.aim.2009.12.005}
  {\path{doi:10.1016/j.aim.2009.12.005}}.

\bibitem{ZM-prod}
G.~Zeller-Meier.
\newblock Produits crois{\'e}s d'une {C*}-alg{\`e}bre par un groupe
  d'automorphismes.
\newblock {\em J. Math. Pures Appl. (9)}, 47:101--239, 1968.

\end{thebibliography}

\end{document}